\documentclass[12pt,centertags,oneside]{amsart}
\usepackage{amsmath,amstext,amsthm,amscd,typearea,hyperref}
\usepackage{amssymb}
\usepackage{a4wide}
\usepackage[mathscr]{eucal}
\usepackage{mathrsfs}
\usepackage{typearea}
\usepackage{charter}
\usepackage{pdfsync}
\usepackage{stmaryrd}
 \usepackage[T1]{fontenc}

\usepackage[a4paper,width=16.5cm,top=3cm,bottom=3cm]{geometry}

\numberwithin{equation}{section}

\usepackage[french,english]{babel}

\allowdisplaybreaks
\emergencystretch=\maxdimen
\usepackage{multicol}

\usepackage{xcolor}

\newtheorem{theorem}{Theorem}[section]

\newtheorem{proposition}[theorem]{Proposition}

\newtheorem{lemma}[theorem]{Lemma}
\newtheorem{remark}[theorem]{Remark}

\newtheorem*{definition*}{Definition}

\newcommand{\cali}[1]{\mathscr{#1}}

\newcommand{\supp}{{\rm supp}}

\newcommand{\Jac}{\mathrm{Jac}}

\renewcommand{\Im}{\mathop{\mathrm{Im}}}

\newcommand{\length}{\mathop{\mathrm{length}}}
\newcommand{\dist}{\mathop{\mathrm{dist}}\nolimits}

\newcommand{\ddc}{{\rm dd^c}}
\newcommand{\dc}{{\rm d^c}}

\newcommand{\dd}{{\rm d}}

\newcommand{\dbar}{\overline\partial}

\newcommand{\ep}{\epsilon}
\newcommand{\vep}{\varepsilon}

\newcommand{\Leb}{{\rm Leb}}

\newcommand{\Vol}{{\rm Vol}}

\newcommand{\Ac}{\cali{A}}
\newcommand{\Bc}{\cali{B}}
\newcommand{\Cc}{\cali{C}}

\newcommand{\Fc}{\cali{F}}

\newcommand{\Oc}{\cali{O}}

\newcommand{\Rc}{\cali{R}}

\newcommand{\W}{\cali{W}}

\newcommand{\cM}{\mathcal{M}}
\newcommand{\cI}{\mathcal{I}}

\newcommand{\FS}{{\rm FS}}
\newcommand{\B}{\mathbb{B}}

\newcommand{\D}{\mathbb{D}}
\newcommand{\C}{\mathbb{C}}

\newcommand{\N}{\mathbb{N}}
\newcommand{\Z}{\mathbb{Z}}
\newcommand{\R}{\mathbb{R}}

\renewcommand\P{\mathbb{P}}

\newcommand{\lp}{\langle}
\newcommand{\rp}{\rangle}

\newcommand{\norm}[1]{\lVert#1\rVert}

\newcommand{\oA}{\mathcal{A}}
\newcommand{\oB}{\mathcal{B}}
\newcommand{\oF}{\mathcal{F}}

\newcommand{\oR}{\mathcal{R}}
\newcommand{\oE}{\mathcal{E}}

\newcommand{\oS}{\mathcal{S}}
\newcommand{\oL}{\mathcal{L}}
\newcommand{\diam}{{\rm diam}}
\newcommand{\oI}{\mathcal{I}}
\newcommand{\oJ}{\mathcal{J}}

\newcommand{\oD}{\mathcal{D}}

\newcommand{\oK}{\mathcal{K}}

\newcommand{\fh}{\mathfrak{h}}

\newcommand{\bH}{\mathbf{H}}

\newcommand{\bP}{\mathbf{P}}



\title{Hole probabilities of random zeros on compact Riemann surfaces}

\author{Hao Wu}
\address{Department of Mathematics,  National University of Singapore - 10, Lower Kent Ridge Road - Singapore 119076}
\email{e0011551@u.nus.edu}

\author{Song-Yan Xie}
\address{Academy of Mathematics and System Science \& Hua Loo-Keng Key Laboratory of Mathematics, Chinese Academy of Sciences, Beijing 100190, China;
	School of Mathematical Sciences, University of Chinese Academy of Sciences, Beijing 100049, China}
\email{xiesongyan@amss.ac.cn}

\date{}
\thanks{}

\begin{document}

\begin{abstract}
We establish a convergence speed estimate for hole probabilities of zeros of random holomorphic sections on compact Riemann surfaces.  
\end{abstract}

\clearpage\maketitle
\thispagestyle{empty}


\medskip


\setcounter{tocdepth}{1}
\tableofcontents\normalfont

\section{Introduction}
Distribution of zeros of random polynomials is a classical subject with rich literature, 
initiated by the works of Waring~\cite{T-book},  Bloch-P\'olya~\cite{BP},
and henceforth developed intensively by the works of
Littlewood-Offord, Hammersley, Kac, Erd\"os-Tur\'an, see~\cite{MR1468385, blo-shi-mrl, MR1308023} for a review and references therein. 

\medskip

In this article, we study the distribution of zeros of random sections of  holomorphic line
bundles on compact Riemann surfaces. More precisely, 
based on the celebrated work~\cite{zel-imrn} of Zelditch (see also~\cite{shi-zel-zre-ind, zei-zel-imrn}) and a recent advance~\cite{DGW-hole-event}, 
we establish a convergence speed estimate of the hole probabilities. Some inspiration also takes from the approach of~\cite{Nishry-Wennman} 
on hole probabilities of Gaussian entire functions.

\medskip

Here is our setting.  
Let $X$ be a compact Riemann surface, equipped with a  K\"ahler form $\omega_0$. Let $\oL$ be a positive holomorphic line bundle on $X$, {\em i.e.}, having  degree $\deg (\oL)>0$. Fix a Hermitian metric  $\fh$ of $\oL$  with strictly positive curvature form $c_1(\oL, \fh)$. We will regard the normalized $(1, 1)$-form $\omega:=c_1(\oL, \fh)/\deg (\oL)$ as a probability measure on $X$, since $\int_X \omega=1$.

\smallskip

 For every positive integer $n$,  the $n$-th power  $\oL^n:=\oL^{\otimes n}$ of the line bundle $\oL$ inherits  a natural metric $\fh_n$ induced by $\fh$, {\em i.e.}, $\norm{s^{\otimes n}(x)}_{\fh_n}:=\norm{s(x)}_\fh^n$ for local sections $s$ of $\oL$ around a point $x\in X$. Let $(\cdot,\cdot)_n$ denote the pointwise Hermitian inner product corresponding to the Hermitian metric $\fh_n$.  On the space $H^0(X,\oL^n)$ of global holomorphic sections of $\oL^n$, we can define a global Hermitian inner product
\begin{equation*}
\lp s_1,s_2\rp_n:=\int_X (s_1,s_2)_n \,\omega_0     \quad  \text{for} \quad  s_1,s_2\in H^0(X,\oL^n).
\end{equation*}
    By the Riemann-Roch theorem,  the global sections  of $\oL^n$ are abundant for $n\gg 1$, that is, $\dim_{\C} H^0(X,\oL^n)= n\cdot \deg (\oL)-g+1\gg 1$. On  the projectivized space $\P H^0(X,\oL^n)$, we consider the  Fubini-Study volume form $V^{\FS}_n$ induced by $\lp \cdot,\cdot\rp_n$. 

\smallskip

The zero set $Z_s$ of a section $s$ in $H^0(X,\oL^n)\setminus \{0\}$, or of an element $[s]$ in $\P H^0(X,\oL^n)$, contains $n\cdot\deg(\oL)$ points counting with multiplicities. Denote by $[Z_s]$ the sum of Dirac masses at the points in $Z_s$.  The
Poincar\'e-Lelong formula states that
$$[Z_s ]=\ddc\log \norm{s}_{\fh_n}+n\cdot\deg (\oL)\, \omega.$$
Here,  $\dc:=\frac{i}{2\pi}(\dbar-\partial)$ and $\ddc=\frac{i}{\pi}\partial\dbar$. 
The \textit{empirical probability measure} of $Z_s$ is given by the unit mass normalization $$\llbracket Z_s \rrbracket:= (n\cdot\deg(\oL))^{-1} [Z_s].$$ 
The distribution of $\llbracket Z_s \rrbracket$ for $s\in H^0(X,\oL^n) \setminus \{0\}$ with respect to the complex Gaussian measure on $H^0(X,\oL^n)$ associated with $\lp \cdot,\cdot\rp_n$
is equal to the distribution of $\llbracket Z_s \rrbracket$ for $[s]\in \P H^0(X,\oL^n)$ with respect to $V^{\FS}_n$~(cf.~\cite[Section 2]{shi-zel-gafa}). 

\medskip
The study of the zero
distribution of random holomorphic sections has  rich literature, cf.~\cite{bay-ind,survey-random-hol,bay-com-mar-TAMS,ble-shi-zel-invent,coman-marin-Nguyen,din-ma-mar-jfa,din-ma-ngu-ens,din-sib-cmh,dre-liu-mar,mar-vu-crelle, MR3519982, shiiffman-jga,shi-zel-cmp,shi-zel-gafa,shi-zel-pamq,shi-zel-zh-jus}.
Notably, specialized in complex dimension one, a celebrated theorem of
Shiffman and Zelditch~\cite{shi-zel-cmp} implies that,
under the distribution $V^{\FS}_n$, the zeros of sections in $\P H^0(X,\oL^n)$ are equidistributed with respect to $\omega$,
{\em i.e.}, for any smooth test function $\phi$ on $X$, there holds
$$ \lim_{n\to \infty} \int_{\P H^0(X,\oL^n)}  \lp \llbracket Z_s \rrbracket, \phi \rp  \, \dd V^{\FS}_n(s)  =\int_X \phi \,\omega.$$

\smallskip

Our goal is to study the hole probabilities of this distribution, in the vein of~\cite{DGW-hole-event,shi-zel-zre-ind, zei-zel-imrn, zel-imrn}. Precisely, for an open subset $D$  in $X$ with $\overline{D}\neq X$, for every large $n$, we consider the hole event 
$$H_{n,D}:=\big\{[s]\in \P H^0(X,\oL^n), Z_s \cap D=\varnothing    \big\}  $$
 about 
 holomorphic sections of $\oL^n$ non-vanishing on $D$.

\medskip

 To state our main result, we need to employ a key functional $\cI_{\omega,D}$ introduced in~\cite{DGW-hole-event}, whose initial model was proposed by~Zeitouni-Zelditch \cite{zei-zel-imrn} (see also~\cite{gho-nis-con, gho-nis-cpam,Nishry-Wennman, zel-imrn}).  Let $\cM(X\setminus D)$ be the set of probability measures supported on $X\setminus D$.  Define 
\begin{equation}
\label{functional I_r}
\cI_{\omega,D}(\mu):=-\int_X U_\mu \, \omega -\int_X U_\mu \, \dd\mu = -\int_X U'_\mu \,\dd \mu +2 \max U'_\mu    ,   \quad \forall\,\mu\in\cM(X\setminus D).
\end{equation}
Here, the function $U_\mu$ (resp.\
$U'_\mu$) is  the $\omega$-potential of type M (resp.\ type I) of $\mu$, see Section~\ref{subsec-potential}.
 When $D$ is non-empty, $\oI_{\omega, D}$ is strictly positive.
  We abbreviate $\oI_D: =\oI_{\omega,  D}$ for convenience, as  $\omega$ is fixed throughout this article. 
One merit of $\oI_D$ is that it admits a unique minimizer $\nu_D \in \cM(X\setminus  D)$, which coincides with the limit  of the conditional expectations $\mathbf{E}\big(\llbracket Z_s \rrbracket \,|\, H_{n, D}\big)$ of the empirical probability measures $\llbracket Z_s \rrbracket$
 of the hole events
  $[s]\in H_{n, D}$ as $n$ tends to infinity.
  A surprising phenomenon is that
   the support of $\nu_D$ avoids an open
  ``forbidden set'' other than $D$  (cf.~\cite{DGW-hole-event}).
  Such phenomenon first appeared in the  setting of Gaussian entire functions~\cite{gho-nis-cpam,Nishry-Wennman}.

\medskip

An open set $E$ in $X$ is said to have smooth boundary $\partial E$, if for every $x\in \partial E$, there is some neighborhood $B$ of $x$ and a smooth diffeomorphism $\psi: E\to \D$, where $\D$ is the unit disc in $\C$, such that 
$$\psi(E\cap B)=\D\cap \{\Im (z) >0\}, \quad    \text{and}\quad    \psi(\partial E  \cap B)=\D\cap \R.$$

For technical convenience, from now on, we assume that the hole $D$ has \textbf{smooth boundary}  and that $\partial D$ \textbf{consists of only finitely many components}.
 Our  main result is about the exponential decay of hole probabilities $\bP_n(H_{n,D})$.

\begin{theorem} \label{thm-main-speed}
There exist a $C_D>0$ independent of $n$, and a $C>0$ independent of $n,D$, such that 
\begin{equation}
    \label{thm 1 estimate}
-\min\oI_D- C_D  {\log n \over n}  \leq {1\over n^2} \log \bP_n(H_{n,D}) \leq -\min\oI_D+C  {\log n \over n}  
\end{equation}
for all sufficiently large $n$.
\end{theorem}

Here, we assume $n\gg 1$ to ensure that $\bP_n(H_{n,D})\neq 0$, since when $X\neq \P^1$ the hole event $H_{n,D}$ can possibly be empty for certain small  $n\in \mathbb{Z}_+$. 
Note that~\eqref{thm 1 estimate} is a strengthen on the convergence speed of the following
hole probability result 
\begin{equation}
    \label{DGW estimate}
\lim_{n\rightarrow \infty}{1\over n^2} \log \bP_n(H_{n,D})
=
-\min\oI_D
\end{equation}
established in~\cite[Lemma 7.2]{DGW-hole-event}.
As a consequence of Theorem~\ref{thm-main-speed}, in dimension one we can recover an estimate established by Shiffman-Zelditch-Zrebiec~\cite{shi-zel-zre-ind} on the exponential decay $\exp(-c_D  n^{t+1})$ of the hole probabilities with respect to a fixed open domain $D$ in a projective manifold of any dimension $t\geq 1$. For the case that $\P^t, t\geq 1$, equipped with the Fubini-Study metric, 
some good asymptotic formulas of
the hole probabilities of polydiscs were established by  Zrebiec~\cite{zre-michigan} and Zhu~\cite{zhu-anapde}.
 Nevertheless, even in the simplest case that  $(X,\oL,\omega)=(\P^1,\Oc(1),\omega_\FS)$, the convergence speed estimate $O(\log n /n)$ in Theorem~\ref{thm-main-speed} is new (cf.~\cite{ zhu-anapde, zre-michigan}).

\medskip

Our second result concerns  asymptotic behavior of the hole probabilities for  discs with varying radii. Fix a   radius $r_0>0$ such that $\overline\B(x,r_0)\neq X$ for any $x\in X$.

\begin{theorem}\label{thm-main-hole}
There exist  constants  $C>c>0$ independent of $x,r$, such that 
\begin{equation}
    \label{thm 2 estimate}
 -C r^4\leq  \lim_{n\to \infty} {1\over n^2} \log \bP_n(H_{n,\B(x,r)})\leq -c  r^4,
 \qquad \forall\, x\in X,\quad
\forall\,0<r\leq r_0.
\end{equation}
\end{theorem}

Previously, 
such hole probabilities with a parameter $r$ on the hole size were not known for zeros of random holomorphic sections,   
except in some very  special cases for which {\em tour de force} calculation is possible, see~\cite{ zhu-anapde, zre-michigan} for such exceptions in dimension one and higher.   
The estimate~\eqref{thm 2 estimate} is inspired by several classical works  in recent decades.
In the  setting of Gaussian  entire functions
$F(z):=\sum_{n=1}^\infty a_n z^n/ \sqrt {n!}$ with the i.i.d.\ standard complex Gaussian  coefficients $a_j$, Sodin-Tsirelson~\cite{sod-tsi-isrj} established that the hole probability with respect to 
a disc of radius $r$ decays as $\exp(-c  r^4)$ when $r$ tends to infinity.  An optimal constant $c$  was obtained by  Nishry~\cite{nis-imrn}.    Recently, Buckley-Nishry-Peled-Sodin \cite{buc-nis-ron-sod-ptrf} studied the hole probabilities for  zeros of hyperbolic Gaussian Taylor series with finite radii of convergence. 
See also~\cite{blo-dau-aop,blo-shi-mrl,gho-nis-con,gho-nis-cpam,cho-zei-imrn,kri-jsp,kur-ska-math-stud,nis-jdm,Nishry-Wennman,per-vir-acta,shi-zel-imrn}.

\medskip
Now we discuss  the insight and ideas in the proofs of Theorems~\ref{thm-main-speed},~\ref{thm-main-hole}. 
First, 
let us explain more about the functional $\oI_D$, since it 
is a crucial ingredient in our approach.  Such a type of functional was first proposed by Zeitouni-Zelditch \cite{zei-zel-imrn} and Zelditch~\cite{zel-imrn} to study large deviations of random zeros of line bundle sections on compact Riemann surfaces.  Later, Ghosh-Nishry \cite{gho-nis-con,gho-nis-cpam} modified it a bit to study the hole event problem for Gaussian entire functions on $\C$.  Recently, Dinh, Ghosh and the first named author \cite{DGW-hole-event} finalized the functional $\oI_D$ to the current version, which is  near but different to Zeitouni-Zelditch's, and used it to solve the equidistribution problem of hole events, and obtained the exponential decay~\eqref{DGW estimate} of the hole probabilities as a by-product.  The proofs in this article heavily rely on the structure theorem of the equilibrium measure obtained in \cite{DGW-hole-event}.

\smallskip

We shall emphasize that~\eqref{thm 1 estimate} is a strengthen of~\eqref{DGW estimate}. For instance, the lower bound of~\eqref{thm 1 estimate} is equivalent to
\begin{equation}
    \label{lower bound}
\bP_n(H_{n,D})\geq
e^{-C_D n\log n}\cdot
e^{-\min\oI_D \cdot n^2},
\qquad
\forall\, n\gg 1, 
\end{equation}
while~\eqref{DGW estimate}  only means  that
\begin{equation}
    \label{previous lower bound}
\bP_n(H_{n,D})\geq
e^{-\epsilon_n n^2}\cdot
e^{-\min\oI_D \cdot n^2},
\qquad
\forall\, n\gg 1,
\end{equation}
for some sequence  $\{\epsilon_n\}_{n\gg 1}$ 
of positive numbers tending to zero. 

\smallskip
The  proofs of both~\eqref{lower bound} and~\eqref{previous lower bound}
boil down to the
following two considerations:
\begin{itemize}
    \item[$(\heartsuit)$] to construct  holomorphic sections $s\in H^0(X,\oL^n)$ 
whose zeros
avoid the hole $D$;
    
    \smallskip\item[$(\diamondsuit)$] 
to show that such $[s]$ constitute a subset of $\mathbb{P}H^0(X,\oL^n)$ with ``large'' volume.
\end{itemize}
As a matter of fact,~\eqref{lower bound} is much more difficult to prove than~\eqref{previous lower bound}. The method of~\cite{DGW-hole-event}
fails to generate enough $[s]\in H_{n, D}$  for~\eqref{lower bound} concerning~$(\diamondsuit)$. 

\smallskip

In general, constructing ``desired'' holomorphic sections
is widely known to be a notoriously difficult task,
and it
lies in the heart of several prominent problems 
in   algebraic and complex geometry 
(cf.~\cite{MR1989197}). 
Let us keep a broader scope in mind that, most of the time, one appeals to three kinds of classical tools:
\begin{itemize}
\item[(1)] methods
motivated by electrostatic potentials;

\item[(2)] differential geometric methods involving positive curvature;

\item[(3)] $L^2$-methods of solving $\overline{\partial}$-equations.
\end{itemize}
Note that (2) and (3)
work in any dimension, while (1) is useful only in dimension one. 

Our approach is a reminiscence of (1). The starting point is  
Zelditch's density formula~\eqref{Zelditch formula}
(cf.~\cite[Theorem 2]{zel-imrn}), whose proof relies on certain bosonization formula~\cite{MR0908551} and ingenious calculation involving theta functions. 
In fact, the functional $\oI_D$
is obtained by exploring the key exponent in the formula~\eqref{Zelditch formula}. Therefore, 
$\oI_D$ serves as a rate function of a large deviation principle (LDP) for the empirical measures $\llbracket Z_s \rrbracket$ of zeros of
random holomorphic sections $[s]\in H_{n, D}$, and such an idea goes back to Dinh, Ghosh and the first author \cite{DGW-hole-event}. 

\medskip

The insight is that
zeros  of a random section in the hole event $H_{n, D}$
shall  repel each other, like electrons. The zero sets $Z_s$ of most $[s]\in H_{n, D}$
 would produce empirical measures
$\llbracket Z_s \rrbracket$ very close to the minimizer $\nu_D$
of the functional $\oI_D$. This is due to the equidistribution result from \cite[Theorem 2]{DGW-hole-event}.

In practice, the upper bound in Theorem~\ref{thm-main-speed} is relatively easy to obtain, by means of Zelditch's density formula~\eqref{Zelditch formula}. 
The formidable  difficulty lies in the  the lower bound, {\em i.e.}~\eqref{lower bound}.  Our key strategies  concerning $(\heartsuit)$ and $(\diamondsuit)$
consist of two steps.
\begin{itemize}
\item[$(\spadesuit)$]
    First, to seek
some 
$[s]\in H_{n, D}$
 such that 
a certain necessary modification 
of $\oI_D$
takes value ``extremely close'' to $ \oI_D(\nu_D)$
on the empirical measure $\llbracket Z_s \rrbracket$. Here ``necessary'' is due to the trouble that 
$\oI_D(\mu)=+\infty$ for any probability measure $\mu$ charging positive mass at some point.

    \smallskip
    \item[$(\clubsuit)$]
 Next,  to show that,
by  ``arbitrary'' mild perturbations $Z_{s'}\subset X\setminus D$ (but still subject to the Abel-Jacobi equation~\eqref{equal-p+q=L-oA}) of the zero set $Z_s$ of $s$, 
we can recover
lots of $[s']\in H_{n, D}$ which constitute ``large'' volume.
\end{itemize}
The reason for requiring the ``extremely closeness'' in $(\spadesuit)$
is to ensure that
after mild perturbations 
$\llbracket Z_{s'} \rrbracket$ still take the modified-functional values    ``very close'' to $ \oI_D(\nu_D)$
so that such $[s']\in H_{n, D}$
could possibly  constitute
``large'' volume
according to Zelditch's density formula~\eqref{Zelditch formula}.

\smallskip
When $g=0$, such a blueprint can be realized directly, see Remark~\ref{rmk-g=0} below. 
However, when $g\geq 1$, one basic difficulty appears  as follows, as pointed out by Zelditch~\cite{zel-imrn}.  Say $n> g$ and  $\deg(\oL)=1$ for simplicity. Note that the zero number of every  $[s]\in \P H^0(X,\oL^n)$ is presicely $n$, which is strictly larger than the dimension $n-g$ of $\P H^0(X,\oL^n)$. Write the zeros of $s$ in any order as $p_1,\dots,p_{n-g}, q_1,\dots, q_g$. 
For a generic $s$, by Abel's theorem, the last $g$ zeros
 $q_1,\dots,q_g$ are uniquely determined by the first $n-g$ zeros  $p_1,\dots,p_{n-g}$. However, even if all the $p_j$'s are outside the hole $D$, it is not certain whether $q_j$'s are in $X\setminus D$ or not,  {\em i.e.}, there seems no direct way to parameterize   $H_{n, D}$.

\smallskip
Let us highlight three major new ingredients along with the structure of this article. 

\begin{itemize}
    \item 
    We adapt the Fekete points theory on compact Riemann surfaces (Section~\ref{sect: Fekete points of the hole event}), inspired by Nishry-Wennman~\cite{Nishry-Wennman}, so that $(\spadesuit)$ can be achieved for every large $n$ (Section~\ref{sect: Functional values with Fekete points}).

    \smallskip
    \item 
    We show the H\"older regularity of $\nu_D$  (Section~\ref{sect: Regularity of the equilibrium measure}) to guarantee that the zeros of any section $s$ obtained in $(\spadesuit)$
are ``sufficiently separated'' with each other.
    \smallskip
    \item 
We introduce a new perturbation method (Section~\ref{sec-hole event section}) to accomplish $(\clubsuit)$ (Section~\ref{sect: Deviation of the two functionals}), in which the key engine is  Lemma~\ref{key lemma}.
\end{itemize}

\smallskip
In the next  Section~\ref{sec:prelim}, we  will present some useful background. 
The proof of Theorem~\ref{thm-main-speed} will be reached 
 in Section~\ref{sect: Decay of the hole probability}, 
and Theorem~\ref{thm-main-hole} in Section~\ref{sect: Asymptotic behavior of the hole probability}.

\medskip

Lastly, we shall emphasize that our approach for Theorem~\ref{thm-main-speed} is not suitable in higher dimension, because it is based on the Abel-Jacobi theory. The  hole probabilities
in higher dimensional varieties remain
a challenging research direction.  

\medskip

\noindent\textbf{Notations:}
The symbols $\lesssim$ and $\gtrsim$ stand for inequalities up to  positive multiplicative constants, which may depend on $X, \omega_0, \omega, D$ but not on any other parameter. The  notation $O(f)$, where $f$ is some given positive function, means some term taking values less than $ C\cdot f$ for some positive constant $C$    which may  depend on $D$.  We denote  by $\D(x,r)\subset \C$ the open disc centered at $x\in \C$ with radius $r>0$. We abbreviate $\D(0,r)$ as $\D_r$,
and $\D_1$ as $\D$. In a Riemann surface or a complex manifold,  $\B(x,r)$ stands for the open ball centered at $x$ having radius $r$ with respect to some given metric.

\medskip
  
\section{Preliminaries} \label{sec:prelim}

\subsection{Abel-Jacobi theory}
 On  a  compact Riemann surface $X$, a \textit{divisor} is a finite formal linear combination $\oD:=\sum a_j x_j$ with $a_j\in \Z$ and $x_j\in X$. If $a_j\geq 0$ for all $j$, then we call $\oD$ an \textit{effective divisor}. The degree $\deg(\oD)$ of $\oD$ means the total sum $\sum a_j$. For a nonzero meromorphic function $f$ on $X$ (resp.\ a nontrivial holomorphic  section $s$ of $\oL$), we write $(f):=(f)_0-(f)_{\infty}$ (resp.\ $(s)$) the divisor defined by the zeros $(f)_0$ and the poles $(f)_{\infty}$ of $f$ (resp.\ zeros of $s$) respectively. 
Such a divisor $\oD=(f)$ is called a \textit{principal divisor}, and it  always has degree $0$.
A  divisor $\mathcal D$ on $X$ can be associated with a line bundle $\Oc(\mathcal D)$.

\medskip

 Take a canonical basis $\alpha_1,\dots,\alpha_g,\beta_1,\dots,\beta_g$ of $H_1(X,\Z)$ with  the intersection numbers $\alpha_i
\cdot \beta_j=\delta_{ij}$  and $\alpha_{j_1}\cdot \alpha_{j_2}=\beta_{j_1}\cdot \beta_{j_2}=0$ for $j_1\neq j_2$. Let $\{\phi_1,\dots,\phi_g\}$ be a basis for the complex vector space of holomorphic differential $1$-forms on $X$. Denote by $\Phi=(\phi_1, \dots,\phi_g)^{\mathsf{T}}$ a column vector of length $g$. The $g\times 2g$ matrix  
$$\Omega:=\Big(\int_{\alpha_1} \Phi, \,\dots\,,  \int_{\alpha_g} \Phi  , \, \int_{\beta_1}\Phi,\,\dots\, ,\, \int_{\beta_g} \Phi\Big).$$
is called a \textit{period matrix}.
The $2g$ columns of  $\Omega$ are linearly independent over $\R$, hence generate a lattice $\Lambda$ in $\C^g$~(cf.~\cite{PAG,gun-book}). Define the \textit{Jacobian variety} of $X$ by
\begin{equation}
    \label{Jacobian variety}
 \mathrm{Jac}(X):= \C^g/ \Lambda,
 \end{equation}
which is a $g$-dimensional complex torus. Fix a base point $p_\star\in X$. The  \textit{Abel-Jacobi map} $\oA:X \to \mathrm{Jac}(X)$ associated with $p_\star$ is defined by
\begin{equation}
    \label{abel-jacobi map}
\oA(x):=\Big(\int_{p_\star}^x \phi_1,\,\dots\,,\,\int_{p_\star}^x \phi_g \Big) \quad  (\text{mod} \,\,\Lambda),
\qquad
\forall x\in X.
\end{equation}
The last vector  is independent of the choices of path from $p_\star$ to $x$ after taking the modulo. 

\smallskip

Let $X^{(m)}$ be the $m$-th symmetric product of $X$. Points in $X^{(m)}$ one-to-one correspond to effective divisors $p_1+\cdots+p_m$ of degree $m$, where  points $p_1,\dots,p_m\in X$ are not necessarily distinct. It is well known that $X^{(m)}$ inherits a complex structure from $X^n$. Let 
\begin{equation}
    \label{def Rm}
\pi_m: X^m\to X^{(m)},\quad   (p_1,\dots,p_m)\mapsto p_1+\cdots+p_m
\end{equation}
be the canonical projection. Let $\Rc_m:=\big\{(p_1,\dots,p_m):\, p_j=p_\ell \text{ for some } j\neq \ell \big\}$ be the ramification subvariety. Near a point $p_1+\cdots+p_m\notin \pi_m(\Rc_m)$, the holomorphic coordinates of  $X^{(m)}$ are given by the local coordinates of $p_1,\dots, p_m\in X$.

\smallskip

For every $t\geq 1$, define a holomorphic map $\oA_t: X^{(t)}\to \mathrm{Jac}(X)$ by 
$$\oA_t\,(p_1+\cdots+p_t):=\sum_{j=1}^t \oA(p_j)= \sum_{j=1}^t\Big(\int_{p_\star}^{p_j} \phi_1,\,\dots \,,\,\int_{p_\star}^{p_j} \phi_g \Big) \quad  (\text{mod} \,\,\Lambda).$$ 
Then $\oA_1=\oA$ is injective, and $\oA_g$ is surjective by the Jacobi inversion theorem.  
Moreover,  Riemann showed that $\oA_g$ is \textit{birational}, {\em i.e.}, there exist two subvarieties $\W\subsetneq X^{(g)}$ and $W_g^1=\oA_t(\W)\subsetneq \Jac(X)$, both of codimension $2$, such that the restriction
$$\oA_g : X^{(g)} \setminus \W \longrightarrow  \mathrm{Jac}(X)  \setminus   W_g^1$$
is biholomorphic. The subvariety $\W$ consists of critical points of $\oA_g$, and $W_g^1$ is called the \textit{Wirtinger subvariety}. Both of them depend on $p_\star$. For  $q_1+\cdots+q_g\in \W$, one has $\dim  H^0(X,\Oc(q_1+\cdots+q_g))\geq 2$~(cf.~\cite{PAG,gun-book}). When $g=1$,  $\W$ is empty.

Abel's theorem states that  $\oA_t(p_1+\cdots+p_t)=\oA_t(p_1'+\cdots+p_t')$ if and only if  $p_1+\cdots+p_t-p_1'-\cdots-p_t'$ is a principal divisor. In particular, there is a holomorphic section $s\in H^0(X, \Oc(p_1+\cdots+p_t))$ with $(s)=p_1'+\cdots+p_t'$.

By the Riemann-Roch theorem, any line bundle  on $X$ of degree $\geq g$ has a nontrivial global section $s$,
hence
is linearly equivalent to $\Oc(\oD)$ for the effective divisor $\oD=(s)$. For large $n$ such that
$m:=n\cdot\deg(\oL)-g>0$,
for any $p_1+\cdots +p_m\in X^{(m)}$, 
we can apply this observation to the line bundle $\oL^n\otimes \Oc(-(p_1+\cdots +p_m))$ of degree $g$ to see that
   there exists some $q_1+\cdots+q_g\in X^{(g)}$ such that 
\begin{equation}\label{equal-p+q=L}
 \Oc(q_1+\cdots+q_g)\simeq \oL^n\otimes \Oc(-(p_1+\cdots +p_m)). 
\end{equation}
Equivalently, if $\oL^n\simeq \Oc(w_1+\cdots+w_{n\cdot\deg(\oL)})$, then 
\begin{equation}\label{equal-p+q=L-oA}
\oA_m(p_1+\cdots+p_m)+\oA_g(q_1+\cdots+q_g)=\oA_n(w_1+\cdots+w_{n\cdot\deg(\oL)})=:\oA_n(\oL^n).
\end{equation}
 Moreover,
the choice of $q_1+\cdots+q_g$ is unique if $\dim H^0(X,\Oc(q_1+\cdots+q_g))=1$, or equivalently,  $\oA_n(\oL^n)-\oA_m(p_1+\cdots+p_m)  \notin W_g^1$.

\smallskip

In this article, we will just assume $\deg(\oL)=1$, because of the following well-known fact, which is a corollary of the Jacobi inversion theorem.

\begin{lemma}[cf.~\cite{DGW-hole-event}--Lemma 6.4] \label{lem-degree1}
	If $\oL'$ is positive line bundle on $X$ of degree $d>1$, then there exists a positive line bundle $\oL$ of degree $1$ such that $\oL^d \simeq \oL '$.
\end{lemma}

\smallskip

\subsection{Quasi-potentials} \label{subsec-potential}

A function $\phi$ on $X$ with values in $\R\cup \{-\infty\}$ is \textit{quasi-subharmonic} if locally it is the difference of a subharmonic function and a smooth one.   If $\phi$ is quasi-subharmonic, then there exists a constant $c\geq 0$ such that $\ddc \phi\geq -c\, \omega$ in the sense of currents. When $c=1$, $\phi$ is called an \textit{$\omega$-subharmonic function}, 
and
$\ddc \phi +\omega$ is a probability measure on $X$ by Stoke's formula.

For any probability measure $\mu$ on $X$, we can write $\mu=\omega+\ddc U_\mu$, where $U_\mu$ is the unique quasi-subharmonic function such that $\max U_\mu=0$. We call $U_\mu$ the \textit{$\omega$-potential of type M} of $\mu$ (M stands for ``maximum''). There is an alternative way to normalize the potential $U'_\mu$ by requiring that $\int_X U'_\nu \,\omega =0$. We call $U'_\mu$ the \textit{$\omega$-potential of type I} of $\mu$ (I stands for ``integration'').
To construct $\omega$-potentials, we use the theory of Green functions.

\begin{lemma} [cf.~\cite{Lang}--Theorem 2.2, \cite{DGW-hole-event}--Lemma 2.1]\label{l:Green}
There exists a \textit{Green function} $G(x,y)$ of $(X,\omega)$ defined on $X\times X$, smooth outside the diagonal, such that for every $x\in X$, one has
$$\int_X G(x,\cdot) \, \omega = 0 \qquad \text{and} \qquad \ddc G(x,\cdot) =\delta_x -\omega,$$
 Moreover, the function $\varrho(x,y):=G(x,y)-\log \dist(x,y)$ is Lipschitz in two variables $x, y$.  
\end{lemma}

For a probability measure $\mu$ on $X$, the aforementioned $\omega$-potentials are given by
\begin{equation}\label{defn-potentil-type-I}
U'_\mu(x):=\int_X G(x,\cdot) \, \dd\mu  \quad\text{and}\quad   U_\mu:=U'_\mu-\max U'_\mu.
\end{equation}

In fact, one can define the $\omega$-potential $U'_{\nu}$ for any positive measure or sign measure $\mu$ on $X$ by the above formula, which will be used in the study of Fekete points.
For any two  positive measures or sign measures $\mu_1,\mu_2$ on $X$, it is clear that $U'_{\mu_1+\mu_2}=U'_{\mu_1}+U'_{\mu_2}$. Moreover, one can show by Stoke's formula that (cf.~\cite[Lemma 2.3]{DGW-hole-event})
\begin{equation}\label{commute-potential}
 \int_X U'_{\mu_1}\,\dd \mu_2=\int_X U'_{\mu_2} \,\dd  \mu_1,\quad
 \text{and}\quad \int_X U'_{\mu_1}\,\dd \mu_1 \leq 0.
\end{equation}

There is another geometric description for the $\omega$-potentials. For every $p\in X$, the holomorphic line bundle $\Oc(p)$ associated with  the one-point divisor $p$ has a canonical section $\mathbf 1_{\Oc(p)}$. 
By the Poincar\'e-Lelong formula and by the $\partial\overline{\partial}$-lemma, we can find  a smooth Hermitian metric $\fh_p$  for every $p$ such that the curvature form is $\omega$ and such that the aforementioned potential reads as 
\begin{equation*}
    \label{geomtric meaning of potential U}
U_{\delta_p}(z)= \log\big\|\mathbf 1_{\Oc(p)}(z) \big\|_{\fh_p}+C_{\fh_p}
\end{equation*}
where $C_{\fh_ p}$ is a normalized constant
subject to the law
$\max {U_{\delta_p}} =0$. 

For every point $\mathbf p:=p_1+\cdots+p_m\in X^{(m)}, m\geq 1$, we denote by \begin{equation}
\label{define delta_p}
\delta_{\mathbf p}
:=
(\delta_{p_1}+\cdots+\delta_{p_m})/m
\end{equation}
the empirical probability measure of $\mathbf{p}$. 
Similarly, we have 
\begin{equation}\label{formula-potential-point}
U_{\delta_{\mathbf p}}(z)={1\over m}\log \big \| {
\mathbf 1_{\Oc(\mathbf p)}(z) \big\|_{\fh_{\mathbf p}}
} +C_{\fh_{\mathbf p}}, 
\end{equation}
where $ \mathbf 1_{\Oc(\mathbf p)}:= \otimes_{j=1}^m  \mathbf 1_{\Oc(p_j)}$, $\fh_{\mathbf p}:=\otimes_{j=1}^m {\fh}_{p_j}$, and $C_{\fh_{\mathbf p}}$ is a normalized constant
subject to the law $\max  U_{\delta_{\mathbf{p}}} =0$.

\medskip

\subsection{Density of zeros}
From now on, we fix a positive line bundle $\oL$ on a compact Riemann surface $X$ with degree $1$. Let $n>g$ be an integer. Set  $m:=n-g$. Define the \textit{exceptional subvariety}
$$\bH_m:=\big \{ p_1+\cdots+p_m\in X^{(m)}:\, \oA_n(\oL^n)-\oA_m(p_1,\dots,p_m)    \in W_g^1 \big\}$$ of $X^{(m)}$ with codimension $2$. 
Outside $\bH_m$,  we have a holomorphic map 
$$\oB_m:X^{(m)}\setminus \bH_m \longrightarrow  X^{(g)}$$
given by  $\oB_m(p_1+\cdots+p_m):=q_1+\cdots+q_g$ in~\eqref{equal-p+q=L} or \eqref{equal-p+q=L-oA}. The image of $\oB_m$ is  $ X^{(g)} \setminus \W$.
Consequently, we obtain a holomorphic map
$$\Ac _m:X^{(m)}\setminus \bH_m  \longrightarrow   \P H^0(X,\oL^n),$$
defined by 
$$\Ac _m (p_1+\cdots+p_m):=[s],$$
where $s\in H^0(X,\oL^n)$ satisfies $(s)= p_1+\cdots+p_m+\oB_m( p_1+\cdots+p_m)$.
The map $\Ac_m$ is generically $\binom{n}{m}$ to one. For a measurable set $A\subset \mathcal M (X)$, if we set
$$A_\P:=  \big\{[s]\in\P H^0(X,\oL^n):\, \llbracket Z_s \rrbracket \in A      \big\},$$
then
\begin{equation}\label{defn-Pn}
\bP_n(A_\P)=\int_{A_\P}\dd V_n^{\FS} =\binom {n}{m}^{-1} \int_{\Ac_m^{-1}(A_\P)}  \Ac_m^*(V_n^{\FS}).
\end{equation}

\medskip

Write $\mathbf q:=\oB_m(\mathbf p)=q_1+\cdots+q_g\in X^{(g)}$. We define
\begin{equation}
    \label{define E(P)}
\oE_m(\mathbf p):={1\over m^2} \sum_{j_1\neq j_2} G(p_{j_1},p_{j_2})
\quad\text{and} \quad
\oF_m(\mathbf p):=\log\big\|  e^{U'_{\delta_{\mathbf p}}+m^{-1}gU'_{\delta_{\mathbf q}}}       \big\|_{L^{2m}(\omega_0)}.
\end{equation}
Here, $G$ is the Green function of $(X, \omega)$, and $U'_{\delta_{\mathbf p}},U'_{\delta_{\mathbf q}}$ are  the $\omega$-potentials of type I. 

\smallskip
Zelditch  \cite[Theorem 2]{zel-imrn} established the following explicit formula for $\Ac_m^*(V_n^\FS)$ (see also \cite[Section 4]{DGW-hole-event}).

\begin{proposition}\label{prop-formula}
For all $n\geq 2g$ and $m=n-g$,
	on $X^{(m)}\setminus \bH_m$,	the positive measures $\Ac _m^*(V_n^{\FS})$  are of the shape
 \begin{equation}
     \label{Zelditch formula}
	C_n  \exp\Big[m^2\Big(\oE_m(\mathbf p)- {2(m+1)\over m} \oF_m(\mathbf p)  \Big) \Big] \kappa_n(\mathbf p).
  \end{equation}
	Here,  $\{C_n\}$ are some normalized positive constants,  and $\kappa_n$ are positive continuous $(m,m)$-forms on $X^{(m)}\setminus \bH_m$ given by
	\begin{align*}
	\kappa_n(\mathbf p):= \Upsilon (\mathbf p) \xi(\mathbf q) \prod_{j=1}^m\prod_{l=1}^g\dist(p_j,q_l)^2   \, i\dd z_1\wedge \dd\overline z_1 \wedge\cdots \wedge i\dd z_m\wedge \dd\overline z_m,
	\end{align*} 
	  where  $e^{-O(m)}\leq\Upsilon (\mathbf p) \leq e^{O(m)}$, and $\xi(\mathbf q)$ is some  strictly positive smooth function on $X^{(g)}$ independent of $n$.
\end{proposition}

Let us explain the above formula. If $\mathbf p\notin \pi_m(\Rc_m) \cup \bH_m$ (see~\eqref{def Rm}),  we can use the uniformizing coordinates $z_j$ about $p_j$ for $j=1, \dots, m$ to construct coordinates  $(z_1, \dots, z_m)$ for $\mathbf p\in X^{(m)}$. For $\mathbf p \in  \pi_m(\Rc_m)\setminus \bH_m$, observing that $\oE_m(\mathbf p)=-\infty$, we set $\Ac_m(V_n^\FS)$ to be $0$ on $\pi_m(\Rc_m)\setminus \bH_m$. 

\smallskip

The following lemma says hole probability does not quite involve with the choice of  $\omega_0$.

\begin{lemma} [\cite{DGW-hole-event}--Lemma 4.2]\label{lem-f-m-p}
	One has
	$$\sup_{\mathbf p \in X^{(m)}}\big| \oF_m(\mathbf p)- \max U'_{\delta_{\mathbf p}}\big|=O\Big({\log m\over m} \Big)  \quad \text{as}\quad m\to \infty,$$
	and $U'_{\delta_{\mathbf p}} \leq C$ for some constant $C$ independent of $m$ and $\mathbf p$.
\end{lemma}

To apply Zelditch's density formula~\eqref{Zelditch formula}  to study hole probabilities, we need to compute its integration over the set $\Ac_m^{-1} (H_{n,D})$. Thus, we put $\oR_{m,D}:=\Ac_m^{-1} (H_{n,D})$, which
can be rephrased as
\begin{equation*}\label{defn-oR-m}
\oR_{m,D}=\big\{ \mathbf p\in (X\setminus  D)^{(m)}\setminus \bH_m: \, \mathbf q=\oB_m(\mathbf p)\in (X\setminus  D)^{(g)} \big\}.
\end{equation*}
 We also define the following separating subset of $\oR_{m,D}$:
\begin{align}\label{defn-oS-m}
\oS_{m,D}:= \big\{\mathbf p\in \oR_{m,D}: \,   \dist &(p_{j_1},p_{j_2})\geq m^{-4}    \,\text{ for all }\, 1\leq j_1\neq j_2\leq m,\\
 &\dist(p_j,q_l)\geq 1 /  m  \, \text{ for all } \, 1\leq j\leq m, 1\leq l\leq g \big\}, \nonumber
\end{align}
where, $\mathbf p:=p_1+\cdots+p_m$ and $\mathbf q:=q_1+\cdots+q_g$. This $\oS_{m,D}$ will play an important role in the proof of Theorem \ref{thm-main-speed}.

\medskip

\section{Regularity of the equilibrium measures}
\label{sect: Regularity of the equilibrium measure}

From \cite[Theorem 1]{DGW-hole-event}, we know that  the minimizer of $\oI_D$ is $\nu_D:= \omega|_{S_D} + \nu_{\partial  D}$, where $S_D:= \{ U_{\nu_D}=0\}\setminus \overline  D$ and $\nu_{\partial  D}$ is a non-vanishing positive measure on $\partial D$. Moreover, if $\nu_D$ has mass on some component of $\partial D$, then $\supp(\nu_D)$ contains this component and $U_{\nu_D}$ is strictly negative on this component of $\partial D$.

In this section, our goal is   to prove the following regularity of $\nu_{\partial D}$, which is used for studying the Fekete points in the next section.

\begin{proposition}\label{prop-regularity}
For any arc $A$ on $\partial D$,
one has $\nu_{\partial  D}(A)\leq c_D \cdot \length (A)$ for some $c_D>0$ independent of $A$.
\end{proposition}

Recall the following result from \cite[Proposition 3.1]{DGW-hole-event}.

\begin{proposition}\label{prop-harmonic-integral}
Let $h$ be a continuous function on $\overline  D$ and harmonic on $ D$. Assume that $h+U_{\nu_D}\leq 0$ on $\overline  D$, then
$$  \int_{\partial  D}  h \,\dd \nu_D \leq 0.$$
\end{proposition}

Suppose that Proposition \ref{prop-regularity} fails, by the following lemma we will construct a harmonic function $h$ such that $h+U_{\nu_D}\leq 0$ on $\overline  D$ while $\int_{\partial  D}  h \,\dd \nu_D >
0$, which is absurd.

\begin{lemma}\label{lem-construct-harmonic}
Let $\mu$ be a probability measure on $\partial \D$ and $z$ be a point in $\partial \D$. Suppose that there exists a  sequence $\{r_j\}_{j\geq 1}$ of positive radii shrinking to $0$ such that $\mu(\D(z,r_j))/r_j$ tends to infinity. Then there exist a positive sequence $\{\beta_j\}_{j\geq 1}$ tending to infinity and a sequence of functions $\{F_j\}_{j\geq 1}$ which are harmonic on $\D$ and continuous on $\partial \D$, such that $\beta_j r_j$ tends to $0$ and such that
for $j$ large enough, $F_j<0$ on $\overline \D\setminus \D(z, \beta_j r_j)$ while   $\int_{\partial \D} F_j \,\dd \mu>0$.
\end{lemma}

\begin{proof}
Denote $\alpha_j:=\mu(\D(z,r_j))$.  Take a sequence $\beta_j:=1/\sqrt{\alpha_j r_j}$  so that  $\beta_j \alpha_j$ tends to infinity while $\beta_j r_j$ tends to zero.  Let $g_j$ be a smooth cut-off function on $\partial\D$ such that
$$  g_j= 1 \,\text{ on } \, \partial\D\cap\D(z,r_j)\quad\text{and} \quad  g_j=0 \,\text{ on } \, \partial\D\setminus \D(z,2r_j).$$
Set $f_j:= g_j-\alpha_j /2$. Let $F_j$ be the harmonic extensions of $f_j$ to $\D$.  It is clear that 
$$ \int_{\partial \D}F_j \,\dd\mu =\int_{\partial \D}g_j \,\dd\mu -\alpha_j /2
\geq
\int_{\partial\D\cap\D(z,r_j)} \mathbf 1 \,\dd\mu -\alpha_j /2
= \alpha_j-\alpha_j/2 >0$$
for large $j$ (in fact,  for every $j$, after a moment of reflection). 

Now we check that $F_j<0$ on $\overline \D\setminus \D(z, \beta_j r_j)$ for large $j$. Since $\{\beta_j\}_{j\geq 1}$
tends to infinity, all but finite $\beta_j\geq 4$. 
In this case, $F_j=-\alpha_j /2<0$ on $\partial\D\setminus \D(z, \beta_jr_j)$.  By the maximum principle, we only need to show that $F_j<0$ on $\partial \D(z,\beta_jr_j)\cap\D$.
The remaining argument is a standard application of the Poisson formula. 

For every $w\in \partial \D(z,\beta_jr_j)\cap\D$, read 
$$ F_j(w)={1\over 2\pi}\int_0^{2\pi} f_j (e^{i\theta}) {1-|w|^2\over |e^{i\theta}-w  |^2} \,\dd \theta = {1\over 2\pi}\int_0^{2\pi} g_j (e^{i\theta}) {1-|w|^2\over |e^{i\theta}-w  |^2} \,\dd \theta -\alpha_j /2. $$
The support of $g_j$ is in $\partial \D\cap \D(z,2r_j)$, whose arc length is  $(4+o(1))r_j\leq 5r_j$ for large $j$. Observe that $1-|w|^2\leq 2(1-|w|)
=
2(|z|-|w|)
\leq
2|z-w|
\leq 2 \beta_j r_j$. 
Moreover, for every 
 $e^{i\theta}$ in $\partial \D\cap \D(z,2r_j)$, we have  $|e^{i\theta}-w | 
\geq
|w-z|-|e^{i\theta}-z|
\geq
\beta_j r_j-2r_j
\geq
\beta_j r_j/2$.   Thus we can conclude that 
$$ F_j(w) \leq {1\over 2\pi}\cdot 5r_j \cdot { 2\beta_j r_j \over (\beta_j r_j/2)^2 }-\alpha_j /2 = {20\over \pi \beta_j}- \alpha_j /2<0$$  for large $j$.
\end{proof}

We have a similar result for the hole $D$.

\begin{lemma}
	Let $\mu$ be a probability measure on $\partial D$ and $z$ be a point in $\partial D$. Suppose that there exists a  sequence $\{r_j\}_{j\geq 1}$ of positive radii shrinking to $0$ such that  $\mu(\B(z,r_j))/r_j$ tends to infinity. 
  Then there exist a positive sequence $\{\beta_j\}_{j\geq 1}$ tending to infinity and a sequence of functions $\{F_j\}_{j\geq 1}$ which are harmonic on $D$ and continuous on $\partial D$, such that $\beta_j r_j$ tends to $0$ and such that
for $j$ large enough,  $F_j<0$ on $\overline D\setminus \B(z, \beta_j r_j)$ while   $\int_{\partial D} F_j \,\dd \mu>0$.
\end{lemma}

\begin{proof}
	If $D$ is simply connected, 
 by the Riemann mapping theorem, we can find
 a biholomorphism $\phi: D\rightarrow \mathbb{D}$. By our assumption on the regularity of $\partial D$, the Kellogg theorem (cf.~\cite[p.~49]{MR1217706}) guarantees that $\phi$
 (resp.\ $\phi^{-1}$) can be continuously 
 extended to the boundary and is Lipschitz on $\overline{D}$ (resp.\ $\overline{\mathbb{D}}$). Thus we can apply Lemma~\ref{lem-construct-harmonic} to conclude.

 \smallskip
 
In general, 	when $D$ is not simply connected,
take a   small $\gamma>0$ such that 
$\B(z,\gamma)$ intersects only one connected component of $D$. Reducing $\gamma$ if necessary, we can draw
 a simply connected subset $D'\subset D$ with  smooth boundary 
such that
$D\cap \B(z,\gamma)\subset D'$.

Denote $\alpha_j:=\mu(\B(z,r_j))$. Take   $\beta_j:=1/\sqrt{\alpha_j r_j}$, so that  $\beta_j \alpha_j$ tends to infinity, while  $\beta_j r_j$ tends to zero. 
Let $g_j$ be  a smooth cut-off function on $\partial D$ such that 
$$  g_j= 1 \,\text{ on } \, \partial D\cap\B(z,r_j)\quad\text{and} \quad  g_j=0 \,\text{ on } \, \partial D\setminus \B(z,2r_j).$$
Set $f_j:= g_j-\alpha_j/2$. Let $F_j, G_j$ be the harmonic extensions of $f_j, g_j$ respectively to $D$,   whose existence is guaranteed by \cite[Theorem I.12]{tsuji-book} since $\partial D$ contains finitely many components. It is clear that $$ \int_{\partial  D}F_j \,\dd\mu =\int_{\partial  D}g_j \,\dd\mu -\alpha_j /2\geq \alpha_j-\alpha_j/2 >0$$ for every $j$. 

Now we prove that $F_j<0$ on $\overline D\setminus \B(z, \beta_j r_j)$ for large $j$
by means of the maximum principle and the auxiliary simply connected subdomain $D'\subset D$. 
By throwing some finite $j$ if necessary, we can assume that all $\beta_j r_j<\gamma$, hence 
$\partial \B(z_j,\beta_j r_j)\cap D$
 coincides with $\partial \B(z_j,\beta_j r_j)\cap D'$ by our choice of $D'$. 
 
 Now we consider the restricted function of $G_j$  on $\partial D'$, denoted by $h_j$.   The key observation is that  the support of $h_j$ is contained in the union of $\partial D'\cap \B(z,2r_j)$ and $\partial D'\setminus \partial D$, while
 these two segments keep relatively large distance compared with $\beta_j r_j$. This feature provides the opportunity for  manipulating the Poisson formula  on $D'$ in the same spirit as using Katutani's formula on $D$. 

By the Riemann mapping theorem and Kellogg's theorem, we can 
 find
 a biholomorphism $\phi: D'\rightarrow \mathbb{D}$ such that  both $\phi$
 and $\phi^{-1}$ can be continuously 
 extended to the boundaries and are $\mathsf{K}$-Lipschitz with some positive constant $\mathsf{K}$. 
 Since $G_j$ solves the Dirichlet problem on $D'$ with the boundary data $h_j$, by the Poisson formula,  we have
 \begin{equation} \label{integral-H}
 G_j(w)={1\over 2\pi}\int_0^{2\pi} h_j\circ \phi^{-1}(e^{i\theta}) {1-|\phi(w)|^2\over |e^{i\theta}-\phi(w)  |^2} \,\dd \theta\end{equation}
 for every $w$ in $\partial \B(z_j,\beta_j r_j)\cap D'$. 
 We only need to integral against two kinds of $\theta$, either $\phi^{-1}(e^{i\theta})\in \partial D'\cap\B(z, 2 r_j)$ or $\phi^{-1}(e^{i\theta})\in \partial D'\setminus \partial D$, since
 $h_j\circ \phi^{-1}(e^{i\theta})$ vanishes otherwise. The first kind of integration contributes $ O(\beta_j^{-1})$ by the same argument as in the proof of Lemma~\ref{lem-construct-harmonic}, using the fact that $\phi$ is Lipschitz. For the second kind of $\theta$ with 
 $\phi^{-1}(e^{i\theta})\in \partial D'\setminus \partial D$, noting that
 the distance between $\partial D'\cap\B(z, \beta_j r_j)$ and $\partial D'\setminus \partial D$ is larger than a fixed positive number  while $\beta_j r_j$ tends to zero, we can show that
$${1-|\phi(w)|^2\over |e^{i\theta}-\phi(w)  |^2} =
O(\beta_j r_j).$$
Indeed, $1-|\phi(w)|^2
\leq 2 ({|\phi(z)|}-|\phi(w)|)
\leq
2|\phi(z)-\phi(w)|
\leq
2\mathsf{K}\beta_jr_j$, while
$|e^{i\theta}-\phi(w)  |^2\geq \mathsf{K}^{-2}\dist(\phi^{-1}(e^{i\theta}),w)^2$, since both $\phi$ and $\phi^{-1}$ are $\mathsf{K}$-Lipschitz. 
Thus
the second kind of integration contributes $ O(\beta_j r_j)$  because $0\leq h_j\leq 1$ by the maximum principle. Summarizing,
for every 
 $w$ in $\partial \B(z_j,\beta_j r_j)\cap D'$, we have  
$G_j(w)=O( \beta_j^{-1}) +O(\beta_j r_j)$.
Thus we can bound $F_j(w)=G_j(w)-\alpha_j /2$ from above by
$$
O(\beta_j^{-1}) +O(\beta_j r_j)
-\alpha_j/2
=
\big(O(\beta_j^{-1}) -\alpha_j/4\big)
+
\big(O( \beta_j r_j)-\alpha_j/4\big).   $$
Both brackets $(\cdot\cdot\cdot)$
are negative for large $j$ by our choice of $\beta_j$. Hence we conclude the proof by the maximum principle.
\end{proof}

\smallskip

\begin{proof}[Proof of Proposition \ref{prop-regularity}]
Assume on the contrary that such $c_D$ does not exist. Then by 
 Vitali's covering  (e.g.~\cite[Chapter 2]{MR1333890}), we can find some point  $z\in \partial  D$ with a sequence $\{r_j\}_{j\geq 1}$ of positive radii tending to zero such that  $\nu_{\partial  D}(\B(z,r_j)) /r_j \to \infty$.  By considering the probability measure $c \nu_{\partial  D}$ for a suitable $c$  and applying Lemma \ref{lem-construct-harmonic}, we can find a positive sequence $\{\beta_j\}_{j\geq 1}$ tending to zero and a sequence of functions $F_j$, which are harmonic on $ D$ and continuous on $\partial  D$, such that $\beta_j r_j$ tends to zero, $F_j<0$ on $\overline  D\setminus \B(z,\beta_j r_j)$, and that $\int_{\partial  D} F_j \,\dd\nu_{\partial  D}>0$ for all $j$.

On the other hand, recall that $U_{\nu_D}$  is strictly negative on the component of $\partial  D$ having positive $\nu_{\partial D}$ measure. In particular, $U_{\nu_D}(z)<0$.  By the upper semi-continuity of $U_{\nu_D}$, when $\beta_j r_j$ is small enough, 
for some small positive constant $\delta$, 
the function $\delta F_j$ satisfies that
$$\delta F_j +U_{\nu_D}\leq 0 \,\text{ on }\, \overline D \quad \text{and}\quad \int_{\partial  D} \delta F_j \,\dd \nu_D> 0.$$
  This contradicts to Proposition \ref{prop-harmonic-integral}.
\end{proof}

Now we can state the regularity result for $U_{\nu_D}$. The proof is standard, which is a consequence of Proposition \ref{prop-regularity}. We provide it here for the convenience of the readers.

\begin{proposition}
    \label{cor-holder}
The quasi-potential $U_{\nu_D}$ is  $1/3$-H\"older.
\end{proposition}

\begin{proof}
Recall the decomposition $\nu_D= \omega|_{S_D} + \nu_{\partial  D}$ and  $U'_{\nu_D}=U'_{\omega|_{S_D}}+U'_{\nu_{\partial  D}}$. It is enough to show that  $U'_{\nu_{\partial  D}}$ is H\"older. Using \eqref{commute-potential}, we have
\begin{align*}
U'_{\nu_{\partial  D}}(z)-U'_{\nu_{\partial  D}}(w)=\int U'_{\nu_{\partial  D}} \,\dd(\delta_z-\delta_w)=\int (U'_{\delta_z}-U'_{\delta_w})\,\dd \nu_{\partial  D}.
\end{align*}
By~\eqref{defn-potentil-type-I}, we obtain that 
$$U'_{\delta_z}(y)-U'_{\delta_w}(y) = \varrho(z,y)-\varrho(w,y)+\log \dist(z,y)-\log\dist(w,y),$$
where  $\varrho$ is Lipschitz. Hence we only need to estimate the integration
$$\int \big(\log \dist(z,y)-\log\dist(w,y)\big)\,\dd \nu_{\partial  D}(y).  $$
Fix a small $a>0$ to be determined. We only need to  consider the case that $\dist(z,w)\leq a$. Let us separate the integration  in two parts as 
$$\Big(\int_{\dist(z,y)>2a} +\int_{\dist(z,y)\leq 2a}\Big) \big(\log \dist(z,y)-\log\dist(w,y)\big)\,\dd \nu_{\partial  D}(y). $$

If $\dist(z,y)>2a$, then $\dist(w,y)>a$. Using Proposition \ref{prop-regularity} and mean value theorem, we can obtain that  
\begin{align*}
&\Big|  \int_{\dist(z,y)>2a}    \big(\log \dist(z,y)-\log\dist(w,y)\big)\,\dd \nu_{\partial  D}(y)\Big| \\
&\leq c'_D  \int_{\dist(z,y)>2a} \dist(z,w) \cdot \dist(w,y)^{-1}\,\dd \Leb(\partial D) \leq 
c'_D a^{-1}
\cdot \dist(z,w),
\end{align*}
where $\Leb(\partial D)$ is the probability measure on $\partial D$ proportional to arc length.

If $\dist(z,y)\leq 2a$, then $\dist(w,y)\leq 3a$.  Using Proposition \ref{prop-regularity} again, we get
\begin{align*}
&\Big|  \int_{\dist(z,y)\leq 2a}    \big(\log \dist(z,y)-\log\dist(w,y)\big)\,\dd \nu_{\partial  D}(y)\Big| \\
&\leq c''_D \Big| \int_{\dist(z,y)\leq 2a} \log \dist(z,y) \,\dd \Leb(\partial D) \Big|+ c''_D\Big| \int_{\dist(w,y)\leq 3a}\log \dist(w,y) \,\dd \Leb(\partial D) \Big| ,
\end{align*}
which is $ O( c''_D a|\log a|)$. Thus, by taking $a:=\sqrt{\dist(z,w)}$, we finish the proof.
\end{proof}

\begin{remark}\rm
Actually we can show that $U_{\nu_D}$ is  Lipschitz.
Nevertheless, any $\gamma$-H\"older continuity of $U_{\nu_D}$ for some $\gamma>0$ suffices  for our purpose. 
Our boundary assumption on $\partial D$ is also not optimal; as long as $\partial D$ is $\Cc^{1+\ep}$ and does not contain cusps, $U_{\nu_D}$ is $\gamma$-H\"older continuous for some $\gamma>0$. To keep this article in an economical length, we prefer not to go further in this direction.
\end{remark}

\medskip

\section{Fekete points of the hole events}
\label{sect: Fekete points of the hole event}

For proving Theorem~\ref{thm-main-speed}, we need to investigate the minimum value of the functional 
$ -\oE_m(\mathbf p) +2 \max U'_{\delta_{\mathbf p}}$
defined on $(X\setminus  D)^{(m)}$. To this aim, we  adapt the theory of logarithmic potentials with external fields (cf.\ \cite[Chapter I]{log-book}) to compact Riemann surfaces.

\medskip

For a closed subset $E$,
its \textit{$\omega$-capacity} is defined as
$$\mathrm{cap}_\omega(E):= \sup_{\mu \in \cM(E)}  \exp\Big[ \int_X U'_\mu \,\dd \mu      \Big],$$
where $U'_\nu$ is the $\omega$-potential of type I of $\nu$, see~\eqref{defn-potentil-type-I}.
If a property holds outside a capacity zero set, then we say that it holds \textit{quasi-everywhere}. It is a fact that if the $\omega$-capacity of $E$ is zero, then the $\omega'$-capacity of $E$ is again zero for any other smooth positive $(1, 1)$-form $\omega'$ on $X$ giving unit area. Another fact is that a capacity zero set is a polar set.

\begin{lemma}
    \label{lem-cap-0}
Let $\mu$ be a probability measure on $X$ such that $\int_X U'_\mu \, \dd\mu$ is finite. If $\mathrm{cap}_\omega(E)=0$, then $\mu(E)=0$.
\end{lemma}

\begin{proof}
Assume that $\mu(E)>0$ on the contrary. Then we can take a probability measure $\mu_E :=c \mu|_E$ for some appropriate $c>0$. 
Note that $\int_X U'_\mu \, \dd\mu> -\infty$
implies that
$\int_X U'_{\mu_E} \, \dd\mu_E> -\infty$.
Thus we have $\mathrm{cap}_\omega (E)\geq \exp\big[\int_X U'_{\mu_E} \, \dd\mu_E \big]>0$, which is absurd.
\end{proof}

Let $Q$ be a real continuous function on $X$.  Define the functional
\begin{equation}
    \label{functional J_Q}
\oJ^Q(\mu): = -\int_X U'_\mu \,\dd \mu +2\int_X Q \,\dd \mu      \end{equation}
for every  $\mu$ in $\cM(X\setminus D)$. 
Here $Q$ is called an \textit{external field}. We have the following analog of \cite[Theorem I.1.3]{log-book}.

\begin{theorem}\label{thm-log-Q}
The functional $\oJ^Q$ admits a unique minimizer $\nu_Q\in \cM(X\setminus D)$ with continuous $\omega$-potential $U'_{\nu_Q}$ of type I. Moreover, one has
\begin{enumerate}
\item $-U'_{\nu_Q} +Q\geq \min\oJ^Q - \int_X Q\,\dd \nu_Q$\, on $X\setminus D$;
\item $-U'_{\nu_Q} +Q= \min\oJ^Q - \int_X Q\,\dd \nu_Q$\, on $\supp(\nu_Q)$.
\end{enumerate}
\end{theorem}

\begin{proof}
The existence of $\nu_Q$ follows from the continuity of $Q$. The uniqueness is based on a well-known fact that
$\mu\mapsto -\int_X U'_\mu \,\dd \mu$ is strictly convex whenever it is finite, see e.g.~\cite[Proposition 12]{zei-zel-imrn}.
Now we verify the remaining two statements. 

\smallskip

We first show that (1) holds on $X\setminus D$  quasi-everywhere.  Suppose, on the contrary, that $E:=\big\{-U'_{\nu_Q} +Q< \min\oJ^Q - \int_X Q\,\dd \nu_Q \big\} \setminus D$ has positive capacity. Then 
\begin{equation}
\label{E_a}
E_a:=\big\{-U'_{\nu_Q} +Q\leq  \min\oJ^Q - \int_X Q\,\dd \nu_Q -a \big\} \setminus D
\end{equation}
also has positive capacity for some small $a>0$. Hence
 we can find some  probability measure $\sigma$ supported on $E_a$ such that $\int_X U'_{\sigma}\,\dd \sigma$  is finite. 
On the other hand, note that 
$$\int_X(-U'_{\nu_Q} +Q)\,\dd \nu_Q
= 
-\int_X U'_{\nu_Q} \,\dd \nu_Q+\int_X Q\,\dd \nu_Q = \min\oJ^Q - \int_X Q\,\dd \nu_Q.$$
Since $-U'_{\nu_Q} +Q$ is lower semicontinuous, there exists a closed ball $B\subset X\setminus D$ such that 
\begin{equation}
    \label{find some ball B}
\quad -U'_{\nu_Q} +Q >  \min\oJ^Q - \int_X Q\,\dd \nu_Q -a/2 \quad \text{on }\, B, 
\quad\text{and}\quad
\nu_Q(B)>0.
\end{equation}

 Set $c:=\nu_Q(B)>0$. Consider the probability measure $\sigma_\vep:= \nu_Q +\vep (c \sigma - \nu_Q|_B)$
 on $X$ for  small $\vep>0$ to be determined. Using~\eqref{commute-potential}, we can rewrite $\oJ^Q(\sigma_\vep)$ as 
\begin{align*}
{}&\,\oJ^Q(\nu_Q) 
-
2
\int_X U'_{\nu_Q}\,\dd( \vep c\sigma-\vep\nu_Q|_B)
+2\int_X Q\,\dd(\vep c\sigma-\vep\nu_Q|_B)
+
O(\vep^2)\\
=
&\,\oJ^Q(\nu_Q)+2\int_X (-U'_{\nu_Q} +Q)\,\dd(\vep c\sigma)-2\int_X (-U'_{\nu_Q} +Q)\,\dd(\vep\nu_Q|_B) +O(\vep^2)\\
\leq
&\,
\oJ^Q(\nu_Q) -2\cdot  a \cdot \vep c +2\cdot a/2 \cdot\vep c+O(\vep^2)
\qquad\ \ 
\text{[use~\eqref{E_a},~\eqref{find some ball B}]}
\\
=&\,\oJ^Q(\nu_Q) -ca\vep  +O(\vep^2). 
\end{align*}
Thus, for $\vep$ small enough, $\oJ^Q(\nu_\vep)<\oJ^Q(\nu_Q)$. This is impossible, because $\nu_Q$ is the minimizer of $\oJ^Q$. 

\smallskip

Next, we prove (2). Since $\int_X U'_{\nu_Q}\,\dd \nu_Q$ is finite, the set $E$ of capacity zero must have $\nu_Q$-measure $0$ by Lemma~\ref{lem-cap-0}. Hence (1) holds $\nu_Q$-almost everywhere. Moreover, $\int_X(-U'_{\nu_Q} +Q)\,\dd \nu_Q= \min\oJ^Q - \int_X Q\,\dd \nu_Q$ imples that (2) holds $\nu_Q$-almost everywhere on $\supp(\nu_Q)$.  The lower semi-continuity of $-U'_{\nu_Q}$ implies (2) holds everywhere on $\supp(\nu_Q)$.

\smallskip
Lastly, we come back to prove (1).  From (2), we see that $U'_{\nu_Q}$ is continuous as a function on $\supp(\nu_Q)$.   By a classical result~\cite[p. 54, 
Theorem III.2]{tsuji-book},  $U'_{\nu_Q}$ is continuous on the whole $X$. Hence (1) holds everywhere on $X\setminus D$.
\end{proof}

\begin{theorem}
 \label{thm-upper-envolep}
One has
$$U'_{\nu_Q} + \min\oJ^Q - \int_X Q\,\dd \nu_Q=\sup_\phi \big\{ \phi \,\, \text{is}\,\,\omega\text{-subharmonic}:\, \phi \leq Q \,\text{ on }\, X\setminus D   \big\}^*. $$
Here $*$ means the upper semicontinuous regularization.
\end{theorem}

\begin{proof}
By Theorem \ref{thm-log-Q}, $U'_{\nu_Q} + \min\oJ^Q - \int_X Q\,\dd \nu_Q$ itself is an $\omega$-subharmonic function which is bounded from above by $Q$ on $X\setminus D$. Hence the direction "$\leq$" is done. 

Denote the function on the right hand side by $V$, which is $\omega$-subharmonic as well. For the probability measure $\nu:=\ddc V+\omega$, we have $\ddc (V-U'_{\nu_Q})=\nu -\nu_Q$. Thus $V-U'_{\nu_Q}$ is subharmonic on $X\setminus \supp(\nu_Q)$. By Theorem~\ref{thm-log-Q}, $Q=U'_{\nu_Q} + \min\oJ^Q - \int_X Q\,\dd \nu_Q$ on $\supp(\nu_Q)$ while
$U'_{\nu_Q}$ is continuous, hence $V=U'_{\nu_Q}+\min\oJ^Q - \int_X Q\,\dd \nu_Q$ on $\supp(\nu_Q)$. By the maximum principle, $V-U'_{\nu_Q}\leq \min\oJ^Q - \int_X Q\,\dd \nu_Q$ on $X\setminus \supp(\nu_Q)$. This shows "$\geq$".
\end{proof}

For each $m>1$ and $\mathbf p\in (X\setminus D)^{(m)}$, define (see~\eqref{define delta_p},~\eqref{define E(P)})
$$ \oK^Q_m(\mathbf p):=-{m\over m-1}\oE_m (\mathbf p) + 2\int_X Q \,\dd \delta_{\mathbf p}.        $$
Let $\mathbf f_m=\sum_{j=1}^m f_{m, j}$ be a minimizer of $\oK^Q_m$. It  may not be unique and we just fix one. 
We call  $\Fc_m:=\{ f_{m, 1},\dots, f_{m, m} \}$ an \textit{$m$-th Fekete set} and call $f_{m, j}$ \textit{Fekete points}. 
Obviously, all the points in $\Fc_m\subset X\setminus D $ are pairwise distinct.

\begin{theorem}
    \label{prop-Fekete}
One has $-U'_{\nu_Q}(x)+Q(x)=\min \oJ^Q -\int_X Q\,\dd \nu_Q$\, for every $x$ in $\Fc_m$.
\end{theorem}

\begin{proof}
By Theorem \ref{thm-log-Q}, we only need to show that $-U'_{\nu_Q}(x)+Q(x)\leq \min \oJ^Q -\int_X Q\,\dd \nu_Q$ every $x$ in $\Fc_m$. Without loss of generality, we take $x=f_{m, 1}$. Define the function $$F(z):=\oK_m^Q(z+f_{m, 2}+\cdots+f_{m, m})$$
for $z\in X\setminus D$,
which attends a minimum value at $z=f_{m, 1}$. Write $\mathbf f'_m:= f_{m, 2}+\cdots+f_{m, m}$. By the definition of $\oE_m$, the new function (see~\eqref{defn-potentil-type-I}) 
$$\widetilde F(z):= -{m\over m-1} {2\over m^2} \sum_{j\geq 2} G(z,f_{m, j}) +{2\over m} Q(z)=
\frac{2}{m}
\Big(-U'_{\delta_{\mathbf f'_m}}(z)+Q(z)\Big)$$
must attend its minimum  on $X\setminus D$ at $z=f_{m, 1}$. In other words,
\begin{equation}\label{ineq-f-1}
-U'_{\delta_{\mathbf f'_m}} (f_{m, 1})+Q(f_{m, 1})
\leq
-U'_{\delta_{\mathbf f'_m}} (z) + Q(z),      \quad \forall\, z\in X\setminus D.  
\end{equation}
Now we consider the auxiliary function $\phi(z):= U'_{\delta_{\mathbf f'_m}}(z)- U'_{\delta_{\mathbf f'}} (f_{m, 1}) +Q(f_{m, 1})$ on $X$. Clearly, it is $\omega$-subharmonic and bounded from above by $Q$ on $X\setminus D$. By Theorem~\ref{thm-upper-envolep}, 
$$\phi(z)\leq U'_{\nu_Q}(z) + \min\oJ^Q - \int_X Q\,\dd \nu_Q,   \quad \forall\,  z\in X.$$
Taking $z=f_{m, 1}$, we obtain the desired inequality.
\end{proof}

Using the theorems above, one can show that $\delta_{\mathbf f_m}$ converges to $\nu_Q$ weakly. Since we do not need this in this article, we left the proof to the readers.
Now we show that any two Fekete points cannot be too close.

\begin{proposition} \label{prop-fekete-seperate}
If  $U'_{\nu_Q}$  is $\gamma$-H\"older for some $0<\gamma\leq 1$, then
there exists a constant $c_Q>0$ independent of $m$ such that 
$$\dist(f_{m, j_1},  f_{m, j_2})\geq c_Q m^{-1/\gamma},  \quad \text{if} \quad j_1\neq j_2.   $$
\end{proposition}

Some arguments in the proof
are inspired by~\cite[Proposition A.2]{Nishry-Wennman}, though
 technically much more involved. One new difficulty is to find appropriate local holomorphic coordinate charts on $X $ to apply the Schwarz lemma. Another new difficulty is to deal with the local weight functions of certain  holomorphic line bundles. The next two lemmas will be useful for handling the mentioned two difficulties respectively. 

 \begin{lemma}
     \label{a variant of Schwarz lemma}
     Let $g: \mathbb{D}\rightarrow \mathbb{C}$
     be a holomorphic function. Let $x, y$ be two points in $\mathbb{D}$. 
     Assume that  $g(y)=0$, $|g(x)|=a$, and
     that $\sup_U |g| <b$ for some constant $a, b>0$, where $U \subset  \mathbb{D}$   is a simply connected region
     containing $x$. Then
     $|x-y|\geqslant \frac{a}{a+b}\cdot \mathrm{dist}(x, \partial U)$. 
 \end{lemma}

\begin{proof}
    Assume on the contrary that
    $|x-y|< \frac{a}{a+b}\cdot \mathrm{dist}(x, \partial U)$.
    Then $y\in U$, and
     $$\mathrm{dist}(y, \partial U)\geqslant 
     \mathrm{dist}(x, \partial U)
     -|x-y|
     \geqslant
     \frac{b}{a+b}\cdot \mathrm{dist}(x, \partial U)>\frac{b}{a}\cdot |x-y|.$$
     Hence the disc $D':=\mathbb{D}(y, \frac{b}{a}\cdot|x-y|)$ centered at $y$ with the radius $\frac{b}{a}\cdot|x-y|$ is contained in $U$. By the maximum principle,  $\sup_ { D'}|g| \leqslant \sup_{ U} |g| <b$.
 Thus by the Schwarz lemma, noting that $g(y)=0$, we obtain that
 $$|g(x)|
 =
 \big|g(y+(x-y))\big|\leqslant \frac{|x-y|}{\frac{b}{a}\cdot|x-y|}\cdot
     \sup_{D'} |g|
     <
     \frac{a}{b}\cdot b=a,$$
     which is absurd.
\end{proof}

 \begin{lemma}
     \label{key trick to change holomorphic sections for new metrics}
     Let $\mathcal{L}$ be a holomorphic line bundle on $X$, and let $U\subset X$ be a simply connected region.
     Let $\fh_1, \fh_2$ be two smooth metrics of $\mathcal{L}$ over $U$, giving the same curvature.
     Then for any holomorphic section $s_1\in \Gamma(U, \mathcal{L})$ of $\mathcal{L}$
     on $U$, one can find a  holomorphic section $s_2\in \Gamma(U, \mathcal{L})$
     with 
     $\|s_2\|_{\fh_2}
     =
    \|s_1\|_{\fh_1} $ on $U$.
 \end{lemma}

 \begin{proof}
    We can define the ratio
     $\fh_2/\fh_1$ at every point of $U$, and check that it is a positive smooth function on $U$. 
     Since the curvature of $\fh_1$ and $\fh_2$ coincide,
$\log(\fh_2/\fh_1)$ is a
harmonic function on $U$.
As $U$ is simply connected, we can write $\log(\fh_2/\fh_1)= \mathsf{Re}(u)$ as the real part of a holomorphic funtion $u\in \Oc(U)$. Hence
$s_2:=e^{u}\cdot s_1$
suffices.
 \end{proof}

\smallskip

\begin{proof}[Proof of Proposition~\ref{prop-fekete-seperate}]
{\bf Step 1.}
By symmetry,  we only need to show the existence of 
$c_Q>0$ and $M_Q\in \mathbb{Z}_+$ such that
$\dist(f_{m, 1},f_{m, 2})\geq c_Q m^{-1/\gamma}$ for all $m\geq M_Q$.  

Write $\mathbf f'_m:= f_{m, 2}+\cdots+f_{m, m}$. 
For $z\in \supp(\nu_Q)$, by Theorems~\ref{thm-log-Q},~\ref{prop-Fekete} and \eqref{ineq-f-1}, we have 
$$U'_{\delta_{\mathbf f'_m}} (z) -    U'_{\delta_{\mathbf f'_m}} (f_{m, 1}) \leq Q(z)-Q(f_{m, 1}) =  U'_{\nu_Q}(z) -U'_{\nu_Q}(f_{m, 1}).   $$
Consequently, $U'_{\delta_{\mathbf f'_m}} (z) - U'_{\nu_Q}(z)    \leq  U'_{\delta_{\mathbf f'_m}} (f_{m, 1})  -U'_{\nu_Q}(f_{m, 1})$ for $z\in \supp(\nu_Q)$.
Since $U'_{\nu_Q}$ is harmonic outside $\supp(\nu_Q)$, 
$U'_{\delta_{\mathbf f'_m}} - U'_{\nu_Q}$ is subharmonic on $X
\setminus\supp(\nu_Q)$.
By the maximum principle,
\begin{equation}\label{U-fm-U-nu-Q}
U'_{\delta_{\mathbf f'_m}} (z) - U'_{\delta_{\mathbf f'_m}} (f_{m, 1})  \leq   U'_{\nu_Q}(z)   -U'_{\nu_Q}(f_{m, 1})   \quad\text{for all}\quad z\in X. 
\end{equation}
Since $U'_{\nu_Q}$ is $\gamma$-H\"older continuous with some constant
$A:=\norm{U'_{\nu_Q}}_{\Cc^\gamma}$, if $\dist(z,f_{m, 1})\leq (m-1)^{-1/\gamma}$, 
the right-hand-side of~\eqref{U-fm-U-nu-Q} is $\leq A\cdot(m-1)^{-1}$.
Hence~\eqref{formula-potential-point}
 implies that
 \[
 \log \big\| \mathbf 1_{\Oc(\mathbf f'_m)}(z)\big\|  _{\fh_{\mathbf f'_m}}   -  \log \big\|\mathbf 1_{\Oc(\mathbf f'_m)}(f_{m, 1})  \big\|_{\fh_{\mathbf f'_m}}     \leq  A\cdot(m-1)^{-1},
 \]
 or equivalently (let $\fh_j$ be the metric of $\Oc(f_{m, j})$)
\begin{equation}\label{dist-z-f1}
\dfrac{
\prod_{j=2}^m
\|\mathbf{1}_{\Oc(f_{m, j})}\|_{\fh_j}(z)
}
{
\prod_{j=2}^m
\|\mathbf{1}_{\Oc(f_{m, j})}\|_{\fh_j}(f_{m, 1})
}
\leq
e^A,\qquad
\forall\,
z\in 
\mathbf{B}_{m}:=\mathbb{B}(f_{m, 1}, (m-1)^{-1/\gamma}).
\end{equation}

\smallskip\noindent
{\bf Step 2.} 
The idea is to show that $f_{m, 2}$ can not be too close to $f_{m, 1}$ by exploring~\eqref{dist-z-f1}. 
First, we  make some  technical preparations which will be essential later.

For every point $y\in X$, we take a local holomorphic chart $g_y: \mathbb{D}_2\rightarrow V_y$ with $g_y(0)=y$.
Since $\{g_y(\mathbb{D})\}_{y\in X}$ covers $X$, by compactness, we can take finitely many points
$y_1, \dots, y_M$ such that
$X$ is covered by the union of $W_{j}:=g_{y_j}(\mathbb{D})$ for $j=1, \dots, M$.

Since each $g_{y_j}: \mathbb{D}_2\rightarrow V_{y_j}$ is a biholomorphism, 
by compactness,
both the restrictions $g_{y_j} | _{\overline{\mathbb{D}}}$ and $g_{y_j}^{-1}{ | _{\overline{W_j}}}$
are $\mathsf K_j$-Lipschitz for some large $\mathsf K_j>0$
with respect to the Euclidean norm on $\mathbb{D}$ and the given distance on $X$.
Set $\mathsf K:=\max\{\mathsf K_1, \dots, \mathsf K_M\}$.

For each $j=1, 2, \dots, m$, we fix a smooth weight function
$\rho_j$ on $V_{y_j}$ satisfying the equation
$\ddc \rho_j=\omega$.
By compactness, we can take some large constant $\mathsf K'>0$ such that all $\rho_j$ are $\mathsf K'$-Lipschitz on $W_j$ for $j=1, 2, \dots, M$.

\smallskip\noindent
{\bf Step 3.}
By compactness argument, there exists some 
  uniform radius $r>0$ such that  for any $z\in X$, the small disc $\mathbb{B}(z, r)$ is contained entirely in one of $W_1, \dots, W_M$.
  Take a large integer $M_Q$ so that $(m-1)^{-1/\gamma}<r$ for all $m
  \geq M_Q$.
  Without loss of generality, we can  assume that $\mathbf{B}_{m}\subset W_1$.

 Since $W_1$ is open, any holomorphic line bundle on $W_1$ is holomorphically trivial (cf.~\cite[p.~229]{Foster}). In particular, for
  each line bundle
  $\Oc(f_{m, j})$, $j=2, \dots, m$, 
  we can fix a  nowhere vanishing section
  $e_j\in \Gamma(W_1, \Oc(f_{m, j}))$ and use it to trivialize
  $\Oc(f_{m, j}) | _{W_1}$
  by identifying
  every holomorphic section
  $s\in \Gamma(W_1, \Oc(f_{m, j}))$
  with the holomorphic function $s/e_j \in \Oc(W_1)$.
  In accordance with this trivialization, 
  the metric $\fh_j | _{W_1}$
  of $\Oc(f_{m, j}) | _{W_1}$ defines a smooth weight function $\varphi_j$ on $W_1$ by the law 
  $$\|s\|_{\fh_j}(z)=|s/e_j|(z)\cdot e^{-\varphi_j(z)},\qquad\forall z\in W_1.$$
  Since the curvature of $\fh_j$ is $\omega$, we  have
  $\ddc \varphi_j=\omega$.

In this way, each canonical section
$\mathbf{1}_{\Oc(f_{m, j})}\in \Gamma(W_1, \Oc(f_{m, j}))$ for $j=2, \dots, m$ has norm
\[
\|\mathbf{1}_{\Oc(f_{m, j})}\|_{\fh_j}(z)=|h_j(z)|\cdot e^{-\varphi_j(z)},
\qquad
\forall z\in W_1,
\]
where $h_j:=\mathbf{1}_{\Oc(f_{m, j})}/e_j$
is some holomorphic functions. 
Hence we can read~\eqref{dist-z-f1} as
\begin{equation}
    \label{key estimate e^A}
\dfrac{ \prod_{j=2}^m |h_j(z)|\cdot \prod_{j=2}^m e^{-\varphi_j (z)}}
{
\prod_{j=2}^m |h_j(f_{m, 1})|\cdot \prod_{j=2}^m e^{-\varphi_j (f_{m, 1})}
}
\leq
e^A, 
\qquad
\forall\,
z\in 
\mathbf{B}_{m}.
\end{equation}

\smallskip\noindent
{\bf Step 4.}
For some technical reason to be seen later, we would like to have some uniform 
Lipschitz bound of each $\varphi_j$ on $W_1$. One way to achieve this is by using some singularity decomposition result (see Lemma~\eqref{l:Green}) of the Green kernel $G(x, y)$ near the diagonal $\{x=y\}\subset X\times X$ and by  choosing appropriate $e_j, \varphi_j$.

Nevertheless, here we present an alternative method.
To apply Lemma~\eqref{key trick to change holomorphic sections for new metrics}, for every $j=2, \dots, m$, we define another metric $\fh_j'$ of the line bundle $\Oc(f_{m, j})$ on
$W_1$,
by evaluating 
any holomorphic section
$s\in \Gamma(W_1, \Oc(f_{m, j}))$
with the norm
$\|s\|_{\fh_j'}(z)=|s/e_j|(z)\cdot e^{-\rho_1(z)}$
at each point $z\in W_1$.
Since $\ddc \rho_1=\omega$, the curvature of $\fh_j'$ coincide with that of $\fh_j$ on $W_1$. By
Lemma~\eqref{key trick to change holomorphic sections for new metrics},
we can find some
holomorphic section $\mathbf{1}_{j}'\in  \Gamma(W_1, \Oc(f_{m, j}))$ with the norm equality
$\|\mathbf{1}_{\Oc(f_{m, j})}\|_{\fh_j}
=
\|
\mathbf{1}_{j}'
\|_{\fh_j'}
$
on $W_1$. 
Equivalent,
we have
\[
|h_j(z)|\cdot e^{-\varphi_j(z)}
=
|\widehat{h_j}(z)|
\cdot e^{-\rho_1(z)},
\qquad
\forall z\in W_1,
\]
where $\widehat{h_j}:=\mathbf{1}_{j}'/e_j\in \Oc(W_1)$ is the  representative holomorphic function of $\mathbf{1}_{j}'$  in the local trivialization.
Thus we can rewrite~\eqref{key estimate e^A} as
\begin{equation}
    \label{key estimate e^A 3}
\frac{\prod_{j=2}\widehat{h}_j(z)}{\prod_{j=2}\widehat{h}_j(f_{m, 1})}
\cdot
e^{\sum_{j=2}^m [
\rho_1(f_{m, 1})-\rho_1(z)
]}
\leq
e^A, 
\qquad
\forall\,
z\in 
\mathbf{B}_{m}.
\end{equation}
Note that by our  preparation, the weight function
$\rho_1$ is $\mathsf K'$-Lipschitz on $W_1$, which contains $\mathbf{B}_m$. 
Hence we have
$$
\sum_{j=2}^m |
\rho_1(f_{m, 1})-\rho_1(z)|
\leq
(m-1)\cdot
\mathsf K'\cdot
(m-1)^{-1/\gamma}
\leq
\mathsf K',
\qquad
\forall\,
z\in 
\mathbf{B}_{m}.
$$
By the above two estimates,
the  holomorphic function 
$$ H:=    \prod_{j=2}^m  {\widehat{h}_j\over \widehat{h}_j(f_{m, 1})},$$
which is well-defined on $W_1$, has absolute norm $\leq e^{A+\mathsf K'}$
on $\mathbf{B}_m$.
 
\smallskip\noindent
{\bf Step 5.}
Assume on the contrary that $\mathrm{dist}(f_{m, 1}, f_{m, 2})< \lambda\cdot (m-1)^{-1/\gamma}$
for some sufficiently small $\lambda$.
It is clear that $f_{m,2}\in  \mathbf{B}_m$, $H(f_{m, 1})=1$, and that $H(f_{m, 2})=0$ because the canonical section $\mathbf{1}_{\Oc(f_{m, 2})}$ vanishes at $f_{m, 2}$ by definition.

Now we are in a position to apply Lemma~\ref{a variant of Schwarz lemma} upon 
$$g:=H\circ g_{y_1},\
x:=g_{y_1}^{-1}(f_{m, 1}),\
y:=g_{y_1}^{-1}(f_{m, 2}), \  
a=1, \  
b:=e^{A+\mathsf K'},\
U:=g_{y_1}^{-1}(\mathbf{B}_m).$$
Since  $g_{y_1}$ is $\mathsf K$-Lipschitz by our preparation, 
$$
(m-1)^{-1/\gamma}
=
 \mathrm{dist}(f_{m, 1}, \partial \mathbf{B}_m)
 =
 \mathrm{dist}(g_{y_1}(x), g_{y_1}(\partial U))
\leq
\mathsf K\cdot
\mathrm{dist}(x, \partial U).
$$
Thus Lemma~\ref{a variant of Schwarz lemma} guarantees that
\[
|x-y|
\geq
\frac{1}{1+e^{A+\mathsf K'}}\cdot \mathrm{dist}(x, \partial U)
\geq
\frac{1}{1+e^{A+\mathsf K'}}\cdot 
\mathsf K^{-1}
\cdot
(m-1)^{-1/\gamma}.
\]
On the other hand,
$g_{y_1}^{-1}$ is also $\mathsf K$-Lipschitz.
Hence
\[
|x-y|
=
|g_{y_1}^{-1}(f_{m, 1})-
g_{y_1}^{-1}(f_{m, 2})
|
\leq
\mathsf K\cdot 
\mathrm{dist}(f_{m, 1}, f_{m, 2}).
\]

Summarizing the above two inequalities,
we obtain
\[
\mathrm{dist}(f_{m, 1}, f_{m, 2})
\geq
\frac{1}{1+e^{A+B}}\cdot 
\mathsf K^{-2}
\cdot
(m-1)^{-1/\gamma}.
\]
Therefore, if $\lambda<\frac{\mathsf K^{-2}}{1+e^{A+\mathsf K'}}$, 
we get a contradiction.
Hence $c_Q=\frac{\mathsf K^{-2}}{1+e^{A+\mathsf K'}}$ suffices.
\end{proof}

\medskip
\section{Functional values with Fekete points}
\label{sect: Functional values with Fekete points}

This goal of this section is to find a  $\mathbf p\in (X\setminus D)^{(m)}$ such that  
\begin{equation}\label{exist-p}
-\oE_m(\mathbf p) +2 \max U'_{\delta_{\mathbf p}} \leq \min\oI_D +O(\log m /m).
\end{equation}
The minimizer of $\oI_D$ is $\nu_D$. While we do not know much about  $-\oE_m(\mathbf p) +2 \max U'_{\delta_{\mathbf p}}$. We will make use of the two functionals $\oJ^Q$ and $\oK^Q_m$ defined in Section \ref{sect: Fekete points of the hole event}. These functionals are linked according through the following result.

\begin{proposition}
\label{vq=vd}
For $Q:=U'_{\nu_D}$, one has  $\nu_Q=\nu_D$.
\end{proposition}

\begin{proof}
Since $\supp(\nu_D)\subset X\setminus D$,  integrating the first inequality in Theorem~\ref{thm-log-Q} against $\nu$, we get
$$-\int U'_{\nu_Q}\,\dd \nu_D +\int U'_{\nu_D} \,\dd \nu_D  \geq \min\oJ^Q - \int_X U'_{\nu_D}\,\dd \nu_Q.  $$
Using \eqref{commute-potential}, we obtain $\min\oJ^Q\leq \int U'_{\nu_D} \,\dd \nu_D$.

On the other hand, $Q$ itself is an $\omega$-subharmonic function and is bounded from above by $Q$ on $X\setminus D$.  By Theorem~\ref{thm-upper-envolep}, we have 
$$
U'_{\nu_D}
\leq 
U'_{\nu_Q} + \min\oJ^Q - \int_X U'_{\nu_D}\,\dd \nu_Q.    $$
Integrating this inequality against $\nu_D$, we get $\int U'_{\nu_D} \,\dd \nu_D\leq \min \oJ^Q$. By the uniqueness of  the minimizer of the functional $\oJ^Q$ (see Theorem~\ref{thm-log-Q}), we conclude that $\nu_Q=\nu_D$.
\end{proof}

The functional $\oE_m(\mathbf p)$ can be viewed as a discrete version of $\int U'_\mu \,\dd\mu$. Thus,
 the functionals $-\oE_m(\mathbf p) +2 \max U'_{\delta_{\mathbf p}}, \oK^Q_m (\mathbf p)$ are   discrete versions of $\oI_D(\mu), \oJ^Q (\mu)$ respectively.  The advantage of introducing $\oJ^Q$ is that under the setting of complex plane $\C$,   $\oJ^Q$ was well-studied, see e.g.\ \cite{log-book}.   Thus, using Proposition \ref{vq=vd}, we will find a $\mathbf p$ satisfying \eqref{exist-p}, from the Fekete points of $\oK^Q_m (\mathbf p)$.

\medskip

From now on, we fix the external field $Q:=U'_{\nu_D}$, which is $1/3$-H\"older continuous by Proposition~\ref{cor-holder}. Hence Proposition~\ref{prop-fekete-seperate} guarantees that distinct Fekete points 
$f_{m, 1}, \dots, f_{m, 
 m}$  are separated for every large $m\gg 1$,
 with a distance estimate $\dist(f_{m, j_1},  f_{m, j_2} )\geq c_Q m^{-3}$ for $j_1\neq j_2$ (we still use $c_Q$ for convenience in this section, which is depending on $D$).

\medskip

The following probability measure $\eta$  will be used two times later.

\begin{lemma}\label{lem-eta}
Let $E$ be a subset  of $X$ (not necessary open) with with finitely many connected components, whose   boundary $\overline{E}\setminus  \mathrm{int}(E)$ consists of disjoint smooth curves. Then there exists a probability measure $\eta$ supported on $\overline{E}$, such that $U'_\eta$ is  constant on $\overline {E}$ and attends its minimum on $\overline{E}$.
\end{lemma}

\begin{proof}
Consider
$$
\Psi:=\sup_\phi \big\{ \phi \,\, \text{is}\,\,\omega\text{-subharmonic}:\, \phi \leq 0 \,\text{ on }\, E   \big\}^*.$$
Let $\eta:=\ddc \Psi+\omega$. The remaining proof goes the same as~\cite[Theorem 21]{zei-zel-imrn}.  This $\eta$ is also the minimizer of the functional
$$\mu \mapsto  -\int_X U'_\mu \,\dd  \mu  \quad\text{for}\quad  \mu\in \mathcal{M}(\overline {E}). $$
The assumptions on the boundary of $E$ is to ensure the continuity of $U'_\eta$.
\end{proof}

Let $\sigma_{m,j}$ be the probability  measure on $\partial \B(f_{m,j}, c_Q m^{-5})$ induced by arc lengths. Define the probability measure 
$$\sigma_m:= {1\over m}\sum_{j=1}^m \sigma_{m,j}.$$
We will use $\sigma_m$  to approximate $\delta_{\mathbf f_m}$, because $\int U'_{\delta_{\mathbf f_m}}\,\dd \delta_{\mathbf f_m}=-\infty$, while $\int U'_{\sigma_m}\,\dd \sigma_m>-\infty$.

\begin{lemma}\label{U-tau-U-f}
There exists some constant $C_D >0$, such that for all sufficiently large $m\gg 1$,
for any $w\in \partial \B(f_{m,j}, c_Q m^{-5})$, one has
$$ \big| U'_{\sigma_m}(w)-U'_{\mathbf f_m}(w) \big|\leq C_D \log m/m. $$  
\end{lemma}

\begin{proof}
We may assume $w\in \partial \B(f_{m,1}, c_Q m^{-5})$.
The argument is by straightforward computation, using~\eqref{defn-potentil-type-I} and Lemma~\ref{l:Green}, $\big| U'_{\sigma_m}(w)-U'_{\mathbf f_m}(w) \big|$ is bounded by
$$\Big| \int \log\dist( w,x )\, \dd(\sigma_m -\mathbf f_m) (x) \Big|  +\Big|\int \varrho(w,x)\, \dd(\sigma_m -\mathbf f_m) (x)  \Big|. $$
Since $\varrho$ is Lipschitz, the second term is $O(m^{-5})$. By definition,  the first term is bounded by
$$ {1\over m} \sum_{\ell=1}^m  \int  \big|  \log\dist( w,x )   -\log\dist(w,f_{m,\ell}) \big|  \, \dd\sigma_{m,\ell}(x).$$
When $\ell \neq 1$, recall the support of $\sigma_{m,\ell}$ is $\partial \B(f_{m,\ell}, c_Q m^{-5})$ and $\dist(w,f_{m,\ell})>c_Qm^{-3}/2$.  By the monotonicity of the $\log$ function, for $x\in \supp(\sigma_{m,\ell})$,
\begin{align*}
\big|  \log\dist( w,x ) &  -\log\dist(w,f_{m,\ell}) \big| \leq \\
\max \Big( &\log{ \dist(w,f_{m, \ell})\over  \dist(w,f_{m, \ell})- c_Q m^{-5} } \,,\, \log{ \dist(w,f_{m, \ell})+ c_Q m^{-5}\over  \dist(w,f_{m, \ell}) }  \Big)
= {O(1)\over m^2}.
\end{align*}
It remains to estimate the first term with $\ell=1$. Since $\sigma_{m,1}$ is the probability measure on $\B(f_{m,1}, c_Q m^{-5})$ induced by arc length, we have 
\begin{align*}\int \big| \log\dist( w,x ) -\log\dist(w, f_{m,1}) \big| \, \dd \sigma_{m,1} (x) &\leq   \int \big| \log\dist( w,x )\big|  \, \dd \sigma_{m,1} (x)+\big|\log (c_Q m^{-5})\big|,
\end{align*}
which is $O(\log m)$ by the singular integration formula
$$ \int    \log |w-z|  \,\dd \Leb_{\partial \D(0, c_Q m^{-5})}(z)  =\log (c_Q m^{-5})   \quad\text{for}\quad z\in\partial \D(0, c_Q m^{-5}), $$
where $\Leb_{\partial \D(0, c_Q m^{-5})}$ is the Lebesgue probability measure on $\partial \D(0, c_Q m^{-5})$.
The proof of the lemma is finished.
\end{proof}

\begin{lemma}
    \label{lem-bound-nu-sigma}
There exists a constant $C_D>0$ independent of $m$ such that
$$ \int U'_{\nu_D}\,\dd \nu_D  -   \int  U'_{\nu_D}  \,\dd  \delta_{\mathbf f_m}  \geq  -C_D {\log m \over m}$$
for all sufficiently large $m\gg 1$.
\end{lemma}

\begin{proof}
By integrating~\eqref{U-fm-U-nu-Q} against $\nu_D=\nu_Q$,  and by using~\eqref{commute-potential}, we obtain
$$ \int U'_{\nu_D} \,\dd\delta_{\mathbf f'_m} -    U'_{\delta_{\mathbf f'_m}} (f_{m, 1}) \leq  \int U'_{\nu_D} \,\dd \nu_D -U'_{\nu_D}(f_{m, 1}).     $$
Hence
\begin{equation}
    \label{step 1 estimate}
\int U'_{\nu_D}\,\dd \nu_D  -   \int  U'_{\nu_D}  \,\dd  \delta_{\mathbf f_m}  \geq  
U'_{\nu_D}(f_{m, 1}) -U'_{\delta_{\mathbf f'_m}} (f_{m, 1}) + \vep_{m,D},  
\end{equation}
where
$$\vep_{m,D}:= \int U'_{\nu_D} \,\dd\delta_{\mathbf f'_m}-  \int  U'_{\nu_D}  \,\dd  \delta_{\mathbf f_m}
=
\frac{1}{m(m-1)}
\cdot \sum_{j=2}^m
U'_{\nu_D}(f_{m, j})
-\frac{1}{m}
\cdot
U'_{\nu_D}(f_{m, 1}) $$
clearly satisfies that $|\vep_{m,D}|\leq 2\cdot m^{-1} \cdot \max |U'_{\nu_D}|= O(m^{-1})$. 
It remains to show that $U'_{\nu_D}(f_{m, 1}) -U'_{\delta_{\mathbf f'_m}} (f_{m, 1})\geq -C_D \log m / m$. 

Let $w\in X$ be a point near $f_{m, 1}$ with $\dist(w,f_{m, 1})= c_Q m^{-5}$. Using~\eqref{defn-potentil-type-I} and  Lemma~\ref{l:Green}, we can bound $\big| U'_{\delta_{\mathbf f'_m}} (f_{m, 1})  - U'_{\delta_{\mathbf f_m}} (w) \big|$ from above by the sum of 
\begin{equation} \label{sum-dist-term}
{1\over m}\big| \log\dist(w,f_{m, 1})\big| + 
\sum_{j=2}^m \Big|
{1\over m}\log \dist(w,f_{m, j})
-
{1\over m-1}\log \dist(f_{m, 1},f_{m, j})
\Big|
\end{equation}
and 
\begin{equation} \label{sum-rho-term}
{1\over m}\big|\varrho(w,f_{m, 1})\big|+\sum_{j=2}^m \Big| 
 {1\over m}\varrho(w,f_{m, j}) - 
 {1\over m-1}\varrho (f_{m, 1},f_{m, j})
 \Big|.
\end{equation}
Since $\varrho$ is Lipschitz, \eqref{sum-rho-term} is  $ O(m^{-1})$.   By the monotonicity of the $\log$ function,
\begin{align}
\big| \log\dist(f_{m, 1},f_{m, j}) &-\log\dist(w,f_{m, j})       \big|  \leq 
\nonumber
\\
\label{log trick}
\max \Big( &\log{ \dist(f_{m, 1},f_{m, j})\over  \dist(f_{m, 1},f_{m, j})- c_Q m^{-5} } \,,\, \log{ \dist(w,f_{m, j})+ c_Q m^{-5}\over  \dist(w,f_{m, j}) }  \Big)
= {O(1)\over m^2},  \end{align}
because $\dist(f_{m, 1},f_{m, j})\geq c_Q m^{-3}$ by Proposition~\ref{prop-fekete-seperate}. Hence \eqref{sum-dist-term} is $ O(\log m /m)$.
Summarizing, 
we have
\begin{equation}
    \label{computations w}
\big| U'_{\delta_{\mathbf f'_m}} (f_{m, 1})  - U'_{\delta_{\mathbf f_m}} (w) \big|= O(\log m/m).
\end{equation}

On the other hand, the $1/3$-H\"older continuity of $U'_{\nu_D}$ yields
$$\big|U'_{\nu_D}(f_{m, 1})-   U'_{\nu_D}(w)\big|
=
O(
m^{-5/3})< 1/m$$
for large $m\gg 1$. Thus we conclude that 
$$   U'_{\nu_D}(f_{m, 1}) -U'_{\delta_{\mathbf f'_m}} (f_{m, 1}) \geq U'_{\nu_D}(w)- U'_{\delta_{\mathbf f_m}}(w)-O(\log m/m)$$
by the above two estimates. 
Hence~\eqref{step 1 estimate} and Lemma \ref{U-tau-U-f} implies that
\begin{equation}
    \label{key inequality about w}
\int U'_{\nu_D}\,\dd \nu_D  -   \int  U'_{\nu_D}  \,\dd  \delta_{\mathbf f_m} \geq U'_{\nu_D}(w)- U'_{\sigma_m}(w)-O(\log m/m).
\end{equation}
By repeating the same argument, or by symmetry, it is clear that~\eqref{key inequality about w} 
 holds for any  $w\in\partial\B(f_{m, j},c_Q m^{-5})$, $j=1, \dots, m$.

 Now take the probability measure $\eta$ in Lemma~\ref{lem-eta} for $E:=\cup_{j=1}^m \partial \B(f_{m, j},c_Q m^{-5})$.
 Integrating~\eqref{key inequality about w}  against $\eta$, we obtain
\begin{align*}
\int U'_{\nu_D}\,\dd \nu_D  -   \int  U'_{\nu_D}  \,\dd  \delta_{\mathbf f_m}  &\geq \int U'_{\nu_D} \,\dd\eta- \int U'_{\sigma_m} \,\dd\eta-O(\log m/m) \\  
\text{[use \eqref{commute-potential}]}
\qquad
&=\int U'_{\eta} \,\dd\nu_D- \int U'_{\eta} \,\dd\sigma_m-O(\log m/m).
\end{align*}
By Lemma~\ref{lem-eta}, $U'_\eta$ equals constantly to its minimum on $E$. Hence $\int U'_{\eta} \,\dd\nu_D\geq \int U'_{\eta} \,\dd\sigma_m$, and   the proof is done.
\end{proof}

\begin{lemma}\label{lem-smooth-sigma}
There exists a constant $C_D>0$  such that, for all sufficiently large $m$, one has  $$    \Big|  \int  U'_{\nu_D}  \,\dd \sigma_m-\int  U'_{\nu_D}  \,\dd  \delta_{\mathbf f_m} \Big|   \leq {C_D \over m} $$ and 
$$
 \oE_m (\mathbf f_m)- \int U'_{\sigma_m}\,\dd\sigma_m \leq C_D {\log m  \over m}.
$$
\end{lemma}

\begin{proof}
 Since $U'_{\nu_D}$ is $1/3$-H\"older continuous with the constant $A:=\norm{U'_{\nu_D}}_{\Cc^{1/3}}$, we have
$$\Big|  \int  U'_{\nu_D}  \,\dd \sigma_m-\int  U'_{\nu_D}  \,\dd  \delta_{\mathbf f_m} \Big|  \leq {1\over m}\sum_{j=1}^m           \Big|  \int  U'_{\nu_D}  \,\dd \sigma_{m,j}-  U'_{\nu_D} (f_{m, j})  \Big| \leq  A \Big({c_Q \over 2m^{5}}\Big)^{1/3},$$
which implies the first inequality.

For the second one, we use Lemma~\ref{l:Green} to rewrite (see~\eqref{define E(P)})
$$\oE_m(\mathbf f_m)= {1\over m^2}\sum_{j\neq k} \Big(\log \dist(f_{m, j},f_{m,k}) +\varrho (f_{m, j},f_{m,k})\Big), $$
and 
$$  \int U'_{\sigma_m}\,\dd\sigma_m={1\over m^2}\sum_{j,k}\int \int \Big(\log\dist(x,y)+\varrho(x,y)  \Big)\,\dd\sigma_{m,j}(x)\dd \sigma_{m,k}(y).  $$

 Denote  by $\Leb_{\partial\D(0,m^{-5})}$   the probability Lebesgue measure on  $\partial\D(0,m^{-5})$. For every $j$, we can estimate that
\begin{align*}
& \Big| \int \int \Big(\log\dist(x,y)+\varrho(x,y)  \Big)\,\dd\sigma_{m,j}(x)\dd \sigma_{m,j}(y) \Big|  \\
&\leq   O(1)
\int  \int  \big|\log|x-y| \big|\,  \dd  \Leb_{\partial\D(0,m^{-5})}(x) \,\dd  \Leb_{\partial\D(0,m^{-5})}(y)+\max\varrho \\
 &=
    O(\log m) +\max\varrho
 +
 O(1),
\end{align*}
where the last inequality is based on the fact that the singular  integration
$$\int\int \log|x-y| \,  \dd  \Leb_{\partial\D(0,m^{-5})}(x) \,\dd  \Leb_{\partial\D(0,m^{-5})}(y)=\log (m^{-5}).   $$

Therefore, to gain the desired estimate $\oE_m (\mathbf f_m)- \int U'_{\sigma_m}\,\dd\sigma_m\leq C_D \log m /m$, we only need to show, for each pair  
 $j\neq k$, that 
\begin{align*}
\Big(\log\dist(f_{m, j},f_{m,k})+\varrho(f_{m, j},f_{m,k})  \Big)-\int \int \Big(\log\dist(x,y)+\varrho(x,y)  \Big)\,\dd\sigma_{m,j}(x)\dd \sigma_{m,k}(y) 
\end{align*}
 is $O(\log m /m)$. 
 First, observe that
the difference involving $\varrho$ is $O(m^{-5})$, because $\varrho$ is Lipschitz while $\sigma_{m,j}$, $\sigma_{m, k}$ are supported on the small circles around $f_{m, j}$ and $f_{m, k}$ respectively. 
For the remaining terms, by the monotonicity of the $\log$ function, we can show that
\begin{align*}
\Big|\log\dist(f_{m, j},f_{m,k})-\int \int \log\dist(x,y)\,\dd\sigma_{m,j}(x)\dd \sigma_{m,k}(y) \Big|
=
O(m^{-2}),
\end{align*}
 because for any 
$x\in \partial\B(f_{m, j},c_Q m^{-5})$, $y\in \partial\B(f_{m, k},c_Q m^{-5})$, we have
\[
\big|\log\dist(f_{m, j},f_{m,k})-
\log\dist(x,y)  \big| 
=
O(m^{-2})
\]
by the same reasoning as~\eqref{log trick}.

Summarizing all the estimates above, we conclude the proof.
\end{proof}

The following is the main result of this section, giving \eqref{exist-p}.

\begin{proposition}\label{prop-fekete-log-bound}
There exists a constant $C_D>0$,  such that
$$-\oE_m(\mathbf f_m)\leq -\int_X U'_{\nu_D}\,\dd \nu_D+ C_D {\log m\over m}$$
and
$$
 \max U'_{\delta_{\mathbf f_m}} \leq \max U'_{\nu_D}  + C_D{\log m\over m} $$
  for all sufficiently large $m$. 
\end{proposition}

\begin{proof}
By symmetry, the inequality~\eqref{ineq-f-1} remains true after replacing  $(f_{m, 1}, \mathbf f'_m)$ by $(f_{m, k}, \sum_{j\neq k}f_{m, j})$  for every $k=1, 2, \cdots, m$. Summing up all these $m$ inequalities, we obtain that (see~\eqref{defn-potentil-type-I})
$$-{1\over m-1}\sum_{j\neq k}  G(f_{m, j},f_{m, k}) +\sum_{j=1}^{m} Q(f_j)\leq -\sum_{j=1}^m G(z,f_{m, j})+mQ(z)   $$
for all $z\in X\setminus D$, where $Q=U'_{\nu_D}$. Multiplying both sides by $(m-1)/m^2$, we receive
\begin{equation}\label{oE-nu-fm} 
-\oE_m(\mathbf f_m) + {m-1\over m}\int U'_{\nu_D} \,\dd \delta_{\mathbf f_m}  \leq  -{m-1\over m} U'_{\delta_{\mathbf f_m}}(z)   +{m-1\over m} U'_{\nu_D}(z).
\end{equation}
 Integrating~\eqref{oE-nu-fm} against the probability measure $\eta$ in Lemma \ref{lem-eta} for $E:=X\setminus  D$, we get
\begin{align*}
-\oE_m(\mathbf f_m) + {m-1\over m}\int U'_{\nu_D} \,\dd \delta_{\mathbf f_m} & \leq  -{m-1\over m} \int U'_{\delta_{\mathbf f_m}} \,\dd \eta   +\int {m-1\over m} U'_{\nu_D} \,\dd\eta \\
\text{[use \eqref{commute-potential}]}
\qquad
&= -{m-1\over m} \int U'_{\eta} \,\dd \delta_{\mathbf f_m}   +{m-1\over m}\int  U'_{\eta} \,\dd\nu_D.
\end{align*}
By Lemma \ref{lem-eta}, $U'_\eta$ is constant on $X\setminus D$. Hence $-\int U'_{\eta} \,\dd \delta_{\mathbf f_m}+\int  U'_{\eta} \,\dd\nu_D=0$. Therefore
\begin{equation}\label{oE-fm-U-nu}
-\oE_m(\mathbf f_m) \leq - {m-1\over m}\int U'_{\nu_D} \,\dd \delta_{\mathbf f_m}.
\end{equation}

Let $\sigma_m$ be the probability measure in Lemma \ref{lem-smooth-sigma}. By Stokes' formula, we have 
$$\int ( U'_{\sigma_m} -U'_{\nu_D})\,\dd(\sigma_m-\nu_D)\leq 0
\qquad\text{(see~\eqref{commute-potential})}.$$  Equivalently,
$$-\int U'_{\sigma_m}\,\dd \sigma_m \geq \int U'_{\nu_D}\,\dd \nu_D -2 \int U'_{\nu_D}\,\dd \sigma_m.   $$
This combining with  Lemma \ref{lem-smooth-sigma} gives
\begin{equation}\label{-oEm-Unur}
-\oE_m (\mathbf f_m) + \int U'_{\nu_D}\,\dd \delta_{\mathbf f_m} +O\Big({\log m\over m}\Big) \geq \int U'_{\nu_D}\,\dd \nu_D  -   \int  U'_{\nu_D}  \,\dd  \delta_{\mathbf f_m}- O\Big ({1\over m}\Big). 
\end{equation}
Taking~\eqref{oE-fm-U-nu} into account, we get
\begin{equation}\label{upper-U-nu-fm}
\int U'_{\nu_D}\,\dd \nu_D  -   \int  U'_{\nu_D}  \,\dd  \delta_{\mathbf f_m}         \leq   -{1\over m} \int U'_{\nu_D}\,\dd \delta_{\mathbf f_m}+C_D'{\log m\over m}. 
\end{equation}
Since $U'_{\nu_D}$ is uniformly bounded, we conclude the first assertion.

\smallskip

For the second one, it is enough to show that $ U'_{\delta_{\mathbf f_m}}(z) \leq  U'_{\nu_D} (z) + C_D \log m/m $ for all $z\in X$. 
Plugging~\eqref{-oEm-Unur} in the left-hand-side of~\eqref{oE-nu-fm}, we obtain that for all $z\in X\setminus D$,
$$ -{1\over m}\int U'_{\nu_D} \,\dd \delta_{\mathbf f_m}  + \int U'_{\nu_D}\,\dd \nu_D  -   \int  U'_{\nu_D}  \,\dd  \delta_{\mathbf f_m} -O\Big({\log m\over m}  \Big)
\leq
{m-1\over m} \Big( U'_{\nu_D} (z) - U'_{\delta_{\mathbf f_m}}(z) \Big).$$
By Lemma \ref{lem-bound-nu-sigma} and the  boundedness of $U'_{\nu_D}$, we deduce that
\[
 U'_{\delta_{\mathbf f_m}}(z) \leq  U'_{\nu_D} (z) + C_D \log m/m ,
 \qquad
\forall z\in X\setminus D.
\]
Noting that $ U'_{\nu_D}-U'_{\delta_{\mathbf f_m}} $ is harmonic on $D$, 
by the maximal principle, the above inequality holds for any $z\in D$ as well. Thus we conclude the proof.
\end{proof}

\medskip

\section{Constructing holomorphic sections by   Fekete points} \label{sec-hole event section}

We have obtained a set of Fekete points  $\Fc_m:=\{ f_{m, 1},\dots, f_{m, m} \}\subset X\setminus D$.
When $X=\P^1$, $\Fc_m$  can be realized directly as the zero set of a degree $m$ homogeneous polynomial in the hole event. However, in the higher genus case, there might not exist $q_1,\dots, q_g$ in $X\setminus D$ such that \eqref{equal-p+q=L-oA} holds for $p_j=f_{m, j}$, due to the restriction of the Abel-Jacobi equation~\eqref{equal-p+q=L-oA}. In other words, this $m$-th Fekete set may not be the zero set of one  section in $H_{n,D}$. For this concern, we will construct a valid section $s_n$ in the hole event  $H_{n,D}$ whose zero set $Z_{s_n}$ deviates from $\Fc_m$ only mildly. 
The following lemma is instrumental in our approach.

\begin{lemma}
    \label{key lemma}
Let $\B(y, R)\subset X$ be the disc
centered at $y\in X$ having radius $R>0$ with respect to $\omega_0$. Then there exist
 small constants $0<\zeta, \beta_1, \beta_2, \beta_3\ll 1$
and a large  integer $k\gg 1$,  such that,
for any $\mathbf Z\in \mathrm{Jac}(X)$, for any $m\geq 1$ points $p_1,\dots, p_{m}\in X$,  
one can find 
$\mathbf P^{[j]}= P_1^{[j]}+\cdots+P_g^{[j]}\in X^{(g)}$ ($1\leq j\leq k$) and $\mathbf Q =Q_1+\cdots+Q_g\in X^{(g)}$
satisfying the following conditions
$$
P_{\ell}^{[j]}, Q_{\ell} \in\,\mathbb{H}_{y,p_1,\dots, p_{m}}^{R,\zeta}:=\B(y,R)\setminus \cup_{j=1}^{m} \B(p_j, \zeta/\sqrt m)\quad\text{for all }\, \ell, j,$$ 
\begin{equation}
    \label{AP+AQ=Z--}
\sum_{j=1}^k\,\oA_g(\mathbf P^{[j]})+\oA_g(\mathbf Q) =\mathbf Z,
\end{equation}
\begin{equation}
\label{dist Q to sub}
\dist \big(\mathbf Q, \pi_g(\Rc_g)\cup \W \big)
\geq
\beta_1,
\end{equation}
\begin{equation} \label{dist-P-P--}
\dist(P_{\ell_1}^{[j_1]},P_{\ell_2}^{[j_2]})
\geq
\beta_2 \quad
\text{for all }\,(j_1, \ell_1)\neq (j_2, \ell_2), 
\end{equation}
\begin{equation} \label{dist-P-Q--}
\dist (P^{[j]}_{\ell_1}, Q_{\ell_2})
\geq
\beta_3 \quad\text{for all } \,j,\ell_1, \ell_2.
\end{equation}
\end{lemma} 

\medskip
The proof is quite involved, and will be separated into $7$ Steps with respect to the thinking process. For clarifying the independence of
some uniform positive constants on $\mathbf{Z}$, $m$, $p_1, \dots, p_m$, we will write
$\mathsf{K}_1, \mathsf{K}_2, \mathsf{K}_3, \dots$ instead of merely $O(1)$. 

\medskip
\begin{proof}
{\bf Step 1.} By demanding $\zeta\ll 1$ to be sufficiently small, we can ensure that the union of any $m$ closed discs
$\cup_{j=1}^{m} \overline{\B}(p_j, \zeta/\sqrt m)$ does not cover the open annulus $\B(y,R)\setminus \overline{\B}(y,2R/3)$ (we can assume that it is nonempty by shrinking $R$ if necessary) by  area comparison
\begin{align}
\sum_{j=1}^m \mathrm{Area}_{\omega_0}
\Big(
\overline{\B}(p_j, \zeta/\sqrt m)
\Big)
&
\leq
m\cdot\mathsf{K}_1\cdot (\zeta/
\sqrt{m})^2
\nonumber
\\
\label{zeta condition 1}
&
=
\mathsf{K}_1\cdot \zeta^2
<
\mathrm{Area}_{\omega_0}
\Big(
\B(y,R)\setminus \overline{\B}(y,2R/3)
\Big).
\end{align}
The existence of a uniform constant $\mathsf{K}_1>0$ is guaranteed by the compactness of $X$. 
Thus 
we can choose $g$ points $Q_{\ell}$ in the nonempty open set
$$
\Big(
\B(y,R)\setminus \overline{\B}(y,2R/3)
\Big)\setminus \cup_{j=1}^{m} \overline{\B}(p_j, \zeta/\sqrt m)$$
for $\ell=1,\dots, g$, 
such that
$\mathbf Q :=Q_1+\cdots+Q_g\notin \pi_g(\Rc_g)\cup \W$,
since $\pi_g(\Rc_g)\cup \W$ is a proper subvariety of $X^{(g)}$. 
However, at the moment it is still not clear why we can choose $\beta_1>0$ independent of $\mathbf{Z}$, $m$, $p_1, \dots, p_m$ such that the requirement~\eqref{dist Q to sub} holds.

\medskip\noindent
{\bf Step 2.}
To this aim, we improve the above reasoning as follows. 
First, we pick $g$ points
$Q_1', \dots, Q_g'$ in 
$\B(y,R)\setminus \overline{\B}(y,2R/3)$
such that
$\mathbf Q' :=Q_1'+\cdots+Q_g'\notin \pi_g(\Rc_g)\cup \W$.
Next,  by continuity argument,  we can choose a sufficiently small positive constant
$\mathsf{K}_2\ll 1$ such that
$\overline{\B}(Q'_{\ell}, \mathsf{K_2})\subset \B(y,R)\setminus \overline{\B}(y,2R/3)$ for $\ell=1, \dots, g$
and such that
$Q_1+\cdots+Q_g\notin \pi_g(\Rc_g)\cup \W$ whenever
$Q_1, \dots, Q_g$
are mild perturbations of
$Q_1', \dots, Q_g'$ respectively within distance $\mathsf{K_2}$, {\em i.e.},
$\mathrm{dist}(Q_{\ell}, Q_{\ell}')\leq \mathsf{K_2}$ for $\ell=1,\dots, g$.
By compactness, for all such $Q_1, \dots, Q_g$ we have~\eqref{dist Q to sub} for some uniform constant $\beta_1>0$. Lastly, we only need to make sure that
$\cup_{j=1}^{m} \overline{\B}(p_j, \zeta/\sqrt m)$ does not cover any of $\overline{\B}(Q_{\ell}, \mathsf{K_2})$ for $\ell=1, \dots, g$. This can be achieved by modifying~\eqref{zeta condition 1} as
\begin{equation}
    \label{zeta condition 2}
    \mathsf{K}_1\cdot \zeta^2
<
\mathrm{Area}_{\omega_0}
\big(
\overline{\B}(Q_{\ell}', \mathsf{K_2})
\big)
\quad
\text{for every}
\quad
\ell=1, \dots, g.
\end{equation}
Thus~\eqref{dist Q to sub} is obtained.
 
\medskip\noindent
{\bf Step 3.} We will choose all the remaining points
$P_{\bullet}^{[\bullet]}$ in $\B(y,R/2)$ so that~\eqref{dist-P-Q--} automatically holds for sufficiently small $\beta_3\ll 1$. Now the only remaining obstacles are~\eqref{AP+AQ=Z--} and~\eqref{dist-P-P--}.

Fix a local holomorphic coordinate  chart $\varphi: \mathbb{D}\hookrightarrow \B(y,R/2)$
with
\begin{equation}
    \label{detail 1}
\|\mathrm{d}\varphi\|
> \mathsf{K}_3>0
\end{equation}
with respect to $\omega_0$ and the Euclidean metric on $\mathbb{D}$. Recall the construction~\eqref{Jacobian variety} of the Jacobian variety defined by the holomorphic $1$-forms $\phi_1, \dots, \phi_g$. 
Read
$$\varphi^*\phi_1(z)=\omega_1(z)\mathrm{d}z, \quad \dots, \quad \varphi^*\phi_g(z)=\omega_g(z)\mathrm{d}z$$ in the local chart, where $z$ is the standard coordinate of the unit disc $\mathbb{D}$.
By the common knowledge of Riemann surfaces (e.g.~\cite[p.~87]{MR1172116}), 
the equation
$\det(\omega_\ell (z_j))_{1\leq \ell , j\leq g}=0$ defines a proper subvariety of $\mathbb{D}^g$ with coordinates $(z_1, \dots, z_g)$. Fix a point
$(z_1', \dots, z_g') \in \mathbb{D}^g$ outside this subvariety, {\em i.e.}, 
$\det(\omega_\ell (z_j'))_{1\leq \ell , j\leq  g}\neq 0$. Therefore the fixed matrix $M:=\big(\omega_\ell (z_j')\big)_{1\leq \ell ,j\leq g}$ is invertible, and the $g$ fixed points $z_1', \dots , z_g'$ are pairwise distinct.

\medskip\noindent
{\bf Step 4.} Fix a small positive constant
\begin{equation}
    \label{delta = ?}
\delta:=\min\big\{|z_{\ell_1}'-z_{\ell_2}'|/3,\, 1-|z_{\ell}'|:\, 1\leq\ell_1< \ell_2\leq g, \,\ell=1, 
\dots, g\big\}>0
\end{equation}
which is obviously independent of $\mathbf{Z}$, $m$, $p_1, \dots, p_m$.
Define a holomorphic map (see~\eqref{abel-jacobi map})
$$F: \mathbb{D}_{\delta}^g\rightarrow\mathrm{Jac}(X),\quad (z_1, 
\dots, z_g)\mapsto\oA\big(\varphi(z_1'+z_1)\big)+\cdots+\oA\big(\varphi(z_g'+z_g)\big).$$ 
Then the  the differential $\mathrm{d}F$ of $F$
at the origin $(0, \dots, 0)\in \mathbb{D}^g$ 
has the Jacobian matrix $M$.
Select another system  
$
(\widetilde\phi_1, \dots, \widetilde\phi_g):=(\phi_1, \dots, \phi_g)\cdot M^{-1} $
of $\mathbb{C}$-linearly independent holomorphic $1$-forms on $X$.
The Jacobian variety $\widetilde{\mathrm{Jac}}(X)$ associated with $(\widetilde\phi_1, \dots, \widetilde\phi_g)$ is clearly 
$\widetilde{\mathrm{Jac}}(X):= \C^g/ \Lambda\cdot M^{-1}$ by~\eqref{Jacobian variety}, and the corresponding Abel-Jacobi map $\widetilde{\oA}: X\rightarrow \widetilde{\mathrm{Jac}}(X)$ defined as
  \[
\widetilde{\oA}(x):=\Big(\int_{p_\star}^x \widetilde\phi_1,\,\dots\,,\,\int_{p_\star}^x \widetilde\phi_g \Big) \quad  (\text{mod} \,\,\Lambda\cdot M^{-1}),
\qquad
\forall x\in X
  \]  
  satisfies that   $\widetilde{\oA}=\oA\cdot M^{-1}$, where in abuse of notation we also denote $M^{-1}$ as the isomorphism from ${\mathrm{Jac}}(X)=\C^g/ \Lambda$ to $\widetilde{\mathrm{Jac}}(X)=\C^g/ \Lambda\cdot M^{-1}$.
  Define another holomorphic map
\begin{equation}
    \label{define F tilde}
\widetilde{F}: \mathbb{D}_{\delta}^g\rightarrow\widetilde{\mathrm{Jac}}(X), \quad (z_1, 
\dots, z_g)\mapsto \widetilde\oA\big(\varphi(z_1'+z_1)\big)+\cdots+\widetilde\oA\big(\varphi(z_g'+z_g)\big).
\end{equation}
It is clear that $\widetilde{F}={F}\cdot M^{-1}$. 
Once we identify both the tangent spaces of $\mathbb{D}_{\delta}^g$ and 
$\widetilde{\mathrm{Jac}}(X)$ at the origins as $\mathbb{C}^g$,
 the   differential $\mathrm{d} \widetilde{F}: \mathbb{C}^g\rightarrow \mathbb{C}^g$ of $\widetilde{F}$
at $(0, \dots, 0)\in \mathbb{D}_{\delta}^g$ 
is  nothing but the identity map.

\medskip\noindent
{\bf Step 5.} We plan to choose
\begin{equation}
    \label{how to choose P}
P^{[j]}_{\ell}:=\varphi(z_\ell'+z^{[j]}_{\ell}) 
\quad\text{for some}\quad z^{[j]}_{\ell}\in \mathbb{D}_{\delta} 
\quad\text{for all}\quad j, \ell.
\end{equation}
Let $\widetilde G:=\widetilde{F}-\widetilde{F}(0, \dots, 0)$
be the translation of 
$\widetilde{F}$ so that
$\widetilde G$ maps the origin of $\mathbb{D}_{\delta}^g$ to that  of $ \widetilde{\mathrm{Jac}}(X)$.
First of all, we transform the equation~\eqref{AP+AQ=Z--}
equivalently to
\begin{equation}
    \label{key equation}
\sum_{j=1}^k\,G(z^{[j]}_{1}, \dots, z^{[j]}_{g})=
\big(\mathbf Z-\oA_g(\mathbf Q)\big)
\cdot M^{-1}
-
k\widetilde{F}(0, \dots, 0)
=:
[\widetilde{\mathbf Z}]\in \widetilde{\mathrm{Jac}}(X)
\end{equation}
for some $\widetilde{\mathbf Z}\in \mathbb{C}^g$ in a fundamental domain 
of $\widetilde{\mathrm{Jac}}(X)=\C^g/ \Lambda\cdot M^{-1}$. Note that the Euclidean norm  of $\widetilde{\mathbf Z}$ is uniformly bounded from above by some finite constant $\mathsf{K}_4$. 

\medskip\noindent
{\bf Step 6.}
The upshot is that the differential $\mathrm{d} \widetilde G: \mathbb{C}^g\rightarrow \mathbb{C}^g$ 
at the origin of $\mathbb{D}_{\delta}^g$ 
is the identity map.
By the inverse function theorem, we can define the inverse $\widetilde G^{-1}$ of $\widetilde G$ on a small neighborhood 
$U$ of the origin of $\widetilde{\mathrm{Jac}}(X)$. The key idea is to show,
    by means of volume comparison, that for some large integer $k$ and for sufficiently small $0<\zeta\ll 1$, we can choose certain $k$ points
$\widetilde{\mathbf{P}}^{[j]}\in \mathbb{C}^g$ near the origin for $j=1, \dots, k$, such that $\sum_{j=1}^k \widetilde{\mathbf{P}}^{[j]}=0$, and 
 such that
\begin{equation}
    \label{how to choose g coordinates}
(z^{[j]}_{1}, \dots, z^{[j]}_{g}):=
\widetilde G^{-1}\big([k^{-1}\cdot \widetilde{\mathbf Z}+\widetilde{\mathbf{P}}^{[j]}] \big)
\end{equation}
meets the requirement~\eqref{dist-P-P--}. 

Now we fulfill some  technical details. For convenience of notation, we regard $U$ as an open neighborhood of the origin of $\mathbb{C}^g$.
It is clear that the differential 
$\mathrm{d} \widetilde G^{-1}: \mathbb{C}^g\rightarrow \mathbb{C}^g$ at the origin
is the identity map.
Hence we can fix a sufficiently small constant $\epsilon>0$, such that for every $\mathbf w$ in the polydisc
$\mathbb{D}_{\epsilon}^g\subset U$, we have
\begin{equation}
    \label{key linearization}
\widetilde G^{-1}(\mathbf w)=\mathbf{w}+O(\|\mathbf w\|^2),
\end{equation}
 where the error term
$O(\|\mathbf w\|^2)$ has Euclidean norm $\leq \mathsf{K}_5\cdot \|\mathbf w\|^2 $.

Denote the finite square lattice
$$\mathcal{W}_{t, \rho}:=\big\{a\rho+b\rho\cdot i \in \mathbb{C}: \, a, b\in 
\mathbb{Z}\cap[1, t]  \big\}$$ 
 for some $t\in \mathbb{Z}_+, \rho>0$ to be determined. Let $\mathfrak I: \mathbb{C}\rightarrow \mathbb{C}^g$ be the diagonal map, sending $z$ to $(z, z, \dots, z)$.
 We will take $k=2  t^{2}$, and take 
the aforementioned points $\widetilde{\mathbf{P}}^{[j]}\in \mathbb{C}^g$ 
for $j=1, \dots, t^{2}$
as mild perturbations
of the  $t^{2}$ points in $\mathfrak I (\mathcal{W}_{t, \rho})$,
and for $j=1+t^{2}, \dots, 2  t^{2}$ as mild perturbations of the  $t^{2}$ points in  $-\mathfrak I (\mathcal{W}_{t, \rho})$.
Here, perturbations mean that the differences are within $\mathbb{D}^g_{\rho/3}$.
In practice, we will take
\begin{equation}
\label{key duality trick}
\widetilde{\mathbf{P}}^{[j+t^{2}]}
=
-
\widetilde{\mathbf{P}}^{[j]} 
\quad
\text{for}\quad
j=1, \dots, t^{2}.
\end{equation}

First of all, we need to ensure that
all possible $\{k^{-1}\cdot \widetilde{\mathbf Z}+\widetilde{\mathbf{P}}^{[j]}\}_{j=1}^{k}$ chosen in this way
are contained in 
$\mathbb{D}_{\epsilon}^g$.
For this goal, we require that 
\begin{equation}
    \label{conditions on k and rho}
\|k^{-1}\cdot \widetilde{\mathbf Z}\|
=
\|2^{-1}\cdot t^{-2}\cdot \widetilde{\mathbf Z}\|
\leq
\epsilon/100    \quad \text{and}\quad
t\cdot\rho
\leq 
\epsilon/2.
\end{equation}

Next, we verify~\eqref{key duality trick} for each $j$.
The only  risk is that one of $g$ coordinates of  $\widetilde G^{-1}
(k^{-1}\cdot \widetilde{\mathbf Z}+\widetilde{\mathbf{P}}^{[j]})$ or of $\widetilde G^{-1}
(k^{-1}\cdot \widetilde{\mathbf Z}+\widetilde{\mathbf{P}}^{[j+t^{2}]})=\widetilde G^{-1}
(k^{-1}\cdot \widetilde{\mathbf Z}-\widetilde{\mathbf{P}}^{[j]})$
lies in $\varphi^{-1}\big(\cup_{j=1}^{m} \B(p_j, \zeta/\sqrt m)\big)=:\Bc$. Note that such ``bad''  choices of $\widetilde{\mathbf{P}}^{[j]}$ are contained either in 
$$\bigcup_{j=0}^g 
\Big(
\widetilde G(\mathbb{D}_{\delta}^j\times \Bc \times \mathbb{D}_{\delta}^{g-1-j})
-
k^{-1}\cdot \widetilde{\mathbf Z}
\Big)
\quad\text{ 
or
in}\quad 
\bigcup_{j=0}^g \Big(
k^{-1}\cdot \widetilde{\mathbf Z}-\widetilde G(\mathbb{D}_{\delta}^j\times \Bc \times \mathbb{D}_{\delta}^{g-1-j})\Big),$$  
whose total Euclidean volume is
$\leq \mathsf{K}_6\cdot \zeta^2$ (see~\eqref{detail 1}).
Since all possible choices of $\widetilde{\mathbf{P}}^{[j]}$
has the Euclidean volume
equal to that of $\mathbb{D}^g_{\rho/3}$, as long as
\begin{equation}
    \label{key condition on zeta}
\mathsf{K}_6\cdot \zeta^2
<
\mathrm{Vol}(\mathbb{D}^g_{\rho/3})
=
\mathsf{K}_7\cdot\rho^{2g},
\end{equation}
we can take $\widetilde{\mathbf{P}}^{[j]}$ not ``bad''.
    Thus~\eqref{AP+AQ=Z--} is  satisfied because of~\eqref{key equation},~\eqref{how to choose g coordinates}  and~\eqref{key duality trick}.

\medskip\noindent
{\bf Step 7.}
In the end, we check the requirement~\eqref{dist-P-P--}.
By the definition~\eqref{how to choose g coordinates}, 
for any $\ell=1, \dots, g$, for any two distinct $j_1, j_2\in \{1, \dots, k\}$,
using~\eqref{key linearization} we can obtain that
\begin{equation}
    \label{distance large}
\aligned
|z^{[j_1]}_{\ell}-z^{[j_2]}_{\ell}|
\geq
(\rho-\rho/3-\rho/3)
-
2\cdot \mathsf{K}_5\cdot
g\epsilon^2
=
\rho/3
-
2g\mathsf{K}_5\cdot \epsilon^2.
\endaligned
\end{equation}
Here $g\epsilon^2$ stands for an upper bound for both $\|k^{-1}\cdot \widetilde{\mathbf Z}+\widetilde{\mathbf{P}}^{[j_1]}\|^2$ and $\|k^{-1}\cdot \widetilde{\mathbf Z}+\widetilde{\mathbf{P}}^{[j_2]}\|^2$. 

Note that under the restriction~\eqref{conditions on k and rho} we can still make 
\begin{equation}
    \label{equation rho, l 2}
   \rho/3
-
2g\mathsf{K}_5\cdot \epsilon^2
\geq
\rho/6.
\end{equation}
Indeed,~\eqref{conditions on k and rho} is a consequence of
\begin{equation}
    \label{explicit values of l, rho}
 2^{-1}\cdot t^{-2} \cdot\mathsf{K}_4
\leq 
\epsilon/100,
\qquad
t\cdot\rho\leq \epsilon/2.
\end{equation}
By shrinking $\epsilon>0$ if necessary, 
we can check that by selecting 
\[
t:=\big[\sqrt{50\mathsf{K}_4/\epsilon}  \,\big]+1,\qquad
\rho:=12g\mathsf{K}_5 \epsilon^2>0,
\]
both the requirements~\eqref{equation rho, l 2}
and~\eqref{explicit values of l, rho} are satisfied. 
Thus~\eqref{distance large} implies that 
$$|z^{[j_1]}_{\ell}-z^{[j_2]}_{\ell}|\geq
\rho/6>0,
\quad\text{hence}
\quad
\mathrm{dist}(P^{[j_1]}_{\ell}, P^{[j_2]}_{\ell})
\geq
\mathsf{K}_8\cdot \rho/6
>0.
$$

On the other hand, 
for any $1\leq j_1, j_2\leq k$,
for any two distinct $\ell_1, \ell_2\in \{1, \dots, g\}$,  by the construction   \eqref{delta = ?}  and  \eqref{how to choose P}, 
we have
\[
\big|(z_{\ell_1}'+z_{\ell_1}^{[j_1]})-
(z_{\ell_2}'+z_{\ell_2}^{[j_2]})\big|\geq \delta,
\quad\text{hence}
\quad
\mathrm{dist}(P_{\ell_1}^{[j_1]}, P_{\ell_2}^{[j_2]})
\geq
\mathsf{K}_8\cdot \delta
>0.
\]

Summarizing, we can take 
$\beta_2:=\min\{\mathsf{K}_8\cdot \rho/6,   \, \mathsf{K}_8\cdot \delta\}$, and demand
$\zeta>0$ to obey~\eqref{zeta condition 2} and~\eqref{key condition on zeta},
so that all the requirements are satisfied. 
\end{proof}

Using Lemma~\ref{key lemma}, we can construct a valid section in $H_{n,D}$, whose zeros are mostly coming from the $m$-th Fekete set.

\begin{proposition}
    \label{cor-fekete-section}
For every $n$ large enough, one can find a holomorphic section $s_n \in H_{n,D}$ with zero set $Z_{s_n}=\{p_1,\dots ,p_m,q_1,\dots,q_g \}$, such that 
\begin{equation} \label{pjfj}
p_j = f_{m, j} \quad\text{for all }\, j\leq m-k_Dg, \end{equation}
\begin{equation} \label{pj1pj2}
\dist(p_{j_1},p_{j_2})\geq  c_D m^{-3}  \quad\text{for all }\,  j_1\neq j_2,  \end{equation}
\begin{equation}\label{pjql}
\dist(p_j,  q_l)\geq c_D/\sqrt m     \quad\text{for all }\,  j,l,
\end{equation}
\begin{equation}\label{qD}
\dist(q_l,D)\geq c_D     \quad\text{for all }\,  l,
\end{equation}
\begin{equation}\label{q-W}
\dist \big(q_1+\cdots+q_g, \pi_g(\Rc_g)\cup \W \big)\geq c_D,
\end{equation}
where $k_D\in\Z_+,c_D>0$ are independent of $n$.
\end{proposition}

\begin{proof}
We apply Lemma \ref{key lemma} upon an open disc $\B(y, R)\Subset X\setminus \overline D$ to get a $k_D:=k\in \N^*$. Now we start with the data
$$p_j:=f_{m, j} \quad\text{for}\quad j\leq m-kg,\quad    \mathbf Z:= \oA_{n} (\oL^n) -\sum_{j=1}^{m-kg} \oA(f_{m, j}),$$
to obtain the remaining $kg+g$ points 
$$p_{m-kg+1}+\cdots+p_m:=\mathbf P^{[1]}+\cdots +\mathbf P^{[k]} \quad\text{and}\quad  q_1+\cdots+q_g:=\mathbf Q.$$
These $n$ points are the zeros   of some holomorphic section $s_n$ because of equation \eqref{AP+AQ=Z--}. The  desired distance inequalities come from the definition of $\mathbb{H}_{y,p_1,\dots, p_{m}}^{R,\zeta}$ and   \eqref{dist Q to sub},\eqref{dist-P-P--}, \eqref{dist-P-Q--}.
\end{proof}

The zeros of $s_n$ also gives an element in $\oS_{m,D}$ (see~\eqref{defn-oS-m}) for all $m$ large enough. The last inequality of the corollary implies that $p_1+\cdots+p_m \notin \bH_m$, {\em i.e.},   $q_1+\cdots +q_g$ is unique.

\medskip

\section{Difference between the two functionals
}
\label{sect: Deviation of the two functionals}

In this section, our goal is to approximate the minimum of the functional $\oI_D$ closely by $-\oE_m(\mathbf p)+2\max U'_{\delta_{\mathbf p}}$, and show that such $\mathbf p\in X^{(m)}$ can constitute ``large enough'' volume.   

\begin{proposition}\label{prop-oE-oF-oI}
There exist  constants $m_0\in \mathbb{Z}_+, C_D>0$, 
such that for every large integer $m\geq m_0$, one can find an open subset  $\Omega_{m,D} \subset \oS_{m,D}\subset X^{(m)}$  (see~\eqref{defn-oS-m}) with volume $\Vol(\Omega_{m,D})\geq m^{-11m}$ and that 
\begin{equation}\label{ineq-oE-oF-oI}
-\oE_m(\mathbf p)+ {2(m+1)\over m} \oF_m(\mathbf p)  \leq \min \oI_D+   C_D { \log m \over m},
\qquad
\forall\, \mathbf p\in \Omega_{m,D}. 
\end{equation}
\end{proposition}

\smallskip

First, we need the following two lemmas, which guarantees the  stability of  
the functional values $-\oE_m(\mathbf p)+2\max U'_{\delta_{\mathbf p}}$ under mild perturbations of $\mathbf{p}:=p_1+\cdots+p_m\in \oS_{m,D}$.

\begin{lemma}\label{lem-p-p'}
For every sufficiently large $m$,
for any $\mathbf p':=p'_1+\cdots+p'_m$ with $p'_j\in \B(p_j,m^{-5})$, one has
$$\big|\oE_m(\mathbf p)-\oE_m(\mathbf p')   \big|= O(m^{-1})  \quad \text{and}\quad 
 \big|\max U'_{\delta_{\mathbf p}}- \max U'_{\delta_{\mathbf p'}} \big|= O( m^{-1}).$$
\end{lemma}

\begin{proof}
By~\eqref{define E(P)} and Lemma~\ref{l:Green}, $\big|\oE_m(\mathbf p)-\oE_m(\mathbf p')   \big|$ is bounded from above by
\begin{align*}
{1\over m^2}\sum_{j\neq \ell}\big| \log\dist(p_j,p_\ell)-\log\dist (p'_j,p'_\ell)  \big|+{1\over m^2} \sum_{j\neq \ell}\big| \varrho(p_j,p_\ell)-\varrho (p'_j,p'_\ell)  \big|.  \label{em-p-p'}
\end{align*}
Since $\varrho$ is Lipschitz, the second term is $O(m^{-5})$. For the first one, using the monotonicity of the $\log$ function, every term $|\cdot\cdot\cdot|$ is  less than
\begin{equation}
    \label{a trick about log}
 \max\Big( \log{ \dist(p_j,p_\ell) \over  \dist(p_j,p_\ell)-2m^{-5} }  \, , \,  \log{  \dist(p_j,p_\ell)+2m^{-5} \over \dist(p_j,p_\ell) } \Big) = O(m^{-1}),
 \end{equation}
because $\dist(p_j,p_\ell) \geq m^{-4}$ for $j\neq \ell$ by~\eqref{defn-oS-m}. Hence their summation ($m(m-1)$ terms), divided by $m^2$, is also $O(m^{-1})$. Thus the first assertion is proved.

\smallskip

For the second assertion, we only need to show the following three estimates: 
\begin{itemize}
\smallskip
    \item[(i)]
$\big|  U'_{\delta_{\mathbf p}}(w)-   U'_{\delta{\mathbf p'}} (w)\big|= O(m^{-1})$,
if $\dist(w,p_j)\geq m^{-6}$ and $\dist(w,p'_j)\geq m^{-6}$;

    \smallskip
    \item[(ii)]
$\max_{X\setminus \B(p_j,m^{-6})}U'_{\delta_{\mathbf p}}\geq \sup_{\B(p_j,m^{-6})}U'_{\delta_{\mathbf p}} -O(m^{-1})$;

    \smallskip
    \item[(iii)]
    $\max_{X\setminus \B(p'_j,m^{-6})}U'_{\delta_{\mathbf p'}}\geq \sup_{\B(p'_j,m^{-6})}U'_{\delta_{\mathbf p'}} -O(m^{-1})$.
\end{itemize}

By~\eqref{defn-potentil-type-I} and Lemma~\ref{l:Green}, the left-hand-side of  (i) is bounded from above by 
\begin{align*}
{1\over m}\sum_{j=1}^m \big| \log\dist(w,p_j)-\log\dist (w,p'_j)  \big|+{1\over m} \sum_{j=1}^m\big| \varrho(w,p_j)-\varrho (w,p'_j)  \big|.
\end{align*}
The remaining argument goes the same as the first assertion above. 

Now we prove (ii) by showing that 
$$\max_{\partial \B(p_j,m^{-6})}U'_{\delta_{\mathbf p}}\geq \sup_{\B(p_j,m^{-6})}U'_{\delta_{\mathbf p}} -O(\log m/m).$$ 
Take any $w\in \partial\B(p_j,m^{-6}),  y\in \B(p_j,m^{-6})$ for e.g.\ $j=1$. By~\eqref{defn-potentil-type-I}, we have
$$  U'_{\delta_{\mathbf p}}(w)-   U'_{\delta_{\mathbf p}} (y)
=
{1\over m}
\big(
G(w,p_1)-G(y,p_1)
\big)
+
{1\over m} 
\sum _{\ell=2}^m  \big( G(w,p_\ell)-G(y,p_\ell)\big) 
    $$
    Using Lemma~\ref{l:Green}
    again, we can estimate
    \[
    \aligned
    {1\over m}\big(
G(w,p_1)-G(y,p_1)\big)
&
=
{1\over m} \big( \log\dist(w, p_1)-\log\dist (y, p_1)  \big)+{1\over m} 
\big(\varrho(w, p_1)-\varrho (y,p_1)  \big)
\\
&
\geq
0 
-
O(1)\dist(w, y)/m
\geq
-O( m^{-7}).
\endaligned
    \]
By similar computations, and using the fact that
\[
\aligned
\big|\log\dist(w, p_\ell)-\log\dist (y, p_\ell)\big|
&
\leq
\log{\dist(p_1, p_\ell)+m^{-6}\over \dist(p_1, p_\ell)-m^{-6} }=
O(m^{-2}),\quad
\forall\, \ell=2, \dots, m, 
\endaligned
\]
we can show that
\[
\sum _{\ell=2}^m  \big| G(w,p_\ell)-G(y,p_\ell)\big|
= O(m^{-1})
\]
Hence  we conclude (ii). The proof of (iii) is likewise. 
\end{proof}

\begin{lemma}\label{lem-modify-g-points}
For every sufficiently large integer $m\gg 1$,
for any $\mathbf p':=p'_1+\cdots+p'_m\in X^{(m)}$ with $\dist(p'_{j_1},p'_{j_2})\geq c_D m^{-3}$ for $j_1 \neq j_2$, and with $p'_j=p_j$ for $1\leq j\leq m-k_D g$, one has
$$\big|\oE_m(\mathbf p)-\oE_m(\mathbf p')   \big|= O(\log m/m)  \quad \text{and}\quad 
 \big|\max U'_{\delta_{\mathbf p}}- \max U'_{\delta_{\mathbf p'}} \big|= O( \log m/m ).$$
\end{lemma}

\begin{proof}
Since $\mathbf p$ and $\mathbf p'$ only differ in the choice of  $k_D g$ points, by~\eqref{define E(P)} and Lemma~\ref{l:Green}, we can bound $\big|\oE_m(\mathbf p)-\oE_m(\mathbf p')   \big|$ from above by
\begin{equation}\label{diff-g-points}
{1\over m^2}\sum_{j=1}^m \sum_{\substack{\ell=m-k_D g+1, \\ \ell\neq j }}^m  
 \Big(\big| \log\dist(p_j,p_\ell)-\log\dist(p'_j,p'_\ell) \big| +\big| \varrho(p_j,p_\ell)-\varrho(p'_j,p'_\ell) \big|  \Big).
\end{equation}
By the assumptions,  each  $\log\dist(\cdot, \cdot)$ is $O(\log m)$. Using the boundedness of $\varrho$, 
 it is clear that~\eqref{diff-g-points} is $ O(k_D \log m/m)$. This proves the first assertion.

\smallskip

For the second one, using (ii) and (iii) in the preceding proof, we only need to show that, if $w\in X$ satisfies $\dist(w,p_j), \dist(w,p'_j)\geq m^{-6}$ for all $j=1,\dots, m$, then
\[
\big|  U'_{\delta_{\mathbf p}}(w)-   U'_{\delta_{\mathbf p'}} (w)\big|= O(\log m/m).
\]   
Applying~\eqref{defn-potentil-type-I} and Lemma~\ref{l:Green} again, we can bound $\big|  U'_{\delta_{\mathbf p}}(w)-   U'_{\delta_{\mathbf p'}} (w)\big|$  from above by 
\begin{align*}
{1\over m}\sum_{j=m-k_D g+1}^m \big| \log\dist(w,p_j)-\log\dist (w,p'_j)  \big|+{1\over m} \sum_{j=m-k_D g+1}^m\big| \varrho(w,p_j)-\varrho (w,p'_j)  \big|.
\end{align*}
The remaining estimate goes the same  as ~\eqref{diff-g-points}. 
\end{proof}

\smallskip

Now we are ready to prove Proposition~\ref{prop-oE-oF-oI}. The idea is to fix a set $\{f_{m, 1}, \dots, f_{m, m}\}$ of Fekete points for every large $m$, and then to apply Proposition~\ref{cor-fekete-section} to show that certain mild perturbations $\mathbf{p}\in X^{(m)}$ of $f_{m, 1}+\cdots+f_{m, m}$, which correspond to zero sets of some sections in $H_{n,D}$, constitute ``large'' volume, and that their functional values  $-\oE_m(\mathbf{p}) +2\max U'_{\delta_{\mathbf p}}$   are close to $\min \oI_D$.

\begin{proof}[Proof of Proposition \ref{prop-oE-oF-oI}] 
Take an $m$-th Fekete set $\{f_{m, 1},\dots, f_{m, m}\}$ for every $m\geq 1$. 
Applying Proposition~\ref{cor-fekete-section} with the same notation, we obtain a holomorphic section $s_n\in H_{n,D}$ whose zero set reads in two separated parts as $\mathbf p:=p_1+\cdots+p_m$ and  $\mathbf q:=q_1+\cdots+q_g$.
Define  $$\Omega_{m,D}:=\otimes_{j=1}^m \big(\B(p_j,m^{-5}) \setminus  D\big) \subset X^{(m)}.$$
Since $D$ has smooth boundary and one half of the unit disc has Euclidean area $\pi /2>1$, the volume of
$
\B(p_j,m^{-5}) \setminus  D$
is $> m^{-10}$ for sufficiently large $m$,
no matter where is $p_j$ in the compact complement $X\setminus D$.
Thus the volume of $\Omega_{m,D}$ is $>m^{-11m}$ for $m\gg 1$.  

\smallskip

Now we check that $\Omega_{m,D} \subset \oS_{m,D}$. Take any point $\mathbf p':=p_1'+\cdots+p_m'\in \Omega_{m,D}$, {\em i.e.}, $p_j'\in \B(p_j,m^{-5})$. Let $\mathbf q'=q_1'\cdots+q_g'$ be a solution to the Abel-Jacobi equation $\oA_m(\mathbf p')+\oA_g(\mathbf q')=\oA_{n}(\oL^{n})$.  We need to show that $\dist(p_{j_1}',p_{j_2}')\geq m^{-4}$ for all $j_1\neq j_2$,  $\dist(p_j',q_l')\geq 1/m$ for all $j,l$, and $q'_l\in X\setminus D$ for all $l$. The first inequality follows from~\eqref{pj1pj2} directly when $m\gg 1$.   For the second and third assertions, since $\oA_1$ is holomorphic, for a fixed Euclidean distance in $\Jac(X)$, 
\begin{align*}
\dist\big(\oA_m(\mathbf p),\oA_m(\mathbf p')\big) &= \dist\Big( \sum_{j=1}^m \oA_1(p_j), \sum_{j=1}^m \oA_1(p'_j) \Big) \\
&\leq \sum_{j=1}^m \dist\big( \oA_1(p_j),\oA_1(p'_j) \big)   
 \leq   \sum_{j=1}^m 
 O(1)\dist(p_j,p'_j)= O(m^{-4}).
\end{align*}
Consequently, $\dist\big(\oA_g(\mathbf q),\oA_g(\mathbf q')\big)\leq O(m^{-4})$. Recalling that $\dist \big(\mathbf q, \pi_g(\Rc_g)\cup \W \big)$ is bounded below by a constant from~\eqref{q-W}, by the inverse function theorem and the compactness argument, 
$\oA_g$ is an isomorphism in a small ball centered at $\mathbf q$ with some uniform radius. For sufficiently large $m$, $\oA_g(\mathbf q')$
is in the image of such a ball under $\oA_g$. Thus $\mathbf q'$ is unique, and
we can read $\mathbf q'$ as $q_1'+\cdots+q_g'$ in a suitable order so that  $\dist(q_l,q_l')\leq O(m^{-4})$. Whence $\dist(p_j',q_l')\geq 1/m$ by \eqref{pjql} and $q'_l\in X\setminus D$ by \eqref{qD}.

\smallskip

It remains to verify \eqref{ineq-oE-oF-oI}.
Now we apply 
Lemma~\ref{lem-modify-g-points} upon the two points $\mathbf p$ and  $\mathbf f_m:=f_{m, 1}+\cdots+f_{m, m}$, which 
clearly  meet the criterion by~\eqref{pjfj}. Taking also Proposition~\ref{prop-fekete-log-bound} into consideration, we see that 
$$ -\oE_m(\mathbf p) +2\max U'_{\delta_{\mathbf p}}\leq \min \oI_D +O(\log m /m).  $$
This combining  with Lemma~\ref{lem-p-p'} yields 
$$  -\oE_m(\mathbf p') +2\max U'_{\delta_{\mathbf p'}}\leq \min \oI_D +O(\log m /m),
\qquad
\forall\,\mathbf p'\in \Omega_{m,D}.$$
Thus we can conclude the proof by Lemma~\ref{lem-f-m-p}.
\end{proof}

\medskip

\section{Decay of the hole probabilities}
\label{sect: Decay of the hole probability}

In this section, we will prove Theorem~\ref{thm-main-speed} in the difficult case  $g\geq 1$ (see Remark~\ref{rmk-g=0} below for the easy case $g=0$). 

Let $\mu$ be a probability measure on $X$. Define
$$\oE_\Delta(\mu):=\int_{(X\times X)\setminus \Delta}G(x,y)\,\dd \mu(x)\dd \mu(y).$$
Observe that $\oE_m(\mathbf p)\leq \oE_\Delta(\delta_{\mathbf p})$ for every $\mathbf{p}=p_1+\cdots+p_m\in X^{(m)}$. The equality holds if and only if   $p_1, \dots, p_m$ are pairwise distinct.

\begin{lemma}\label{lem-oE-oI}
	For any open set $D$ (can be empty) in $X$, one has
	$$\inf_{\mu \in \cM(X\setminus D)}\big(- \oE_\Delta(\mu) +2\max U'_\mu \big) =\min \oI_{D}=\oI_D(\nu_D).$$
\end{lemma}

\begin{proof}
 Recall~\eqref{defn-potentil-type-I} that 
	$$\oI_{D}(\mu)= - \int_{X \times X} G(x,y) \,\dd \mu(x) \dd\mu(y) +2\max U'_\mu.  $$  
		Noting that $G(x,y)=-\infty$ on the diagonal $\Delta:=\{x=y\}$ of 
 $X\times X$, it is clearly that 
   $\nu_D\times \nu_D$ charges no mass on $\Delta$, for otherwise
  $ \oI_{D}(\nu_D)=+\infty$! Hence we conclude the proof. 
\end{proof}

\smallskip

\begin{proof}[Proof of  Theorem~\ref{thm-main-speed}]
By~\eqref{defn-Pn} and Proposition \ref{prop-formula}, we can compute
\begin{align*}
\bP_n (H_{n,D})&= \binom{n}{m}^{-1}\int_{\oR_{m,D}}  \Ac_m^*(V_{n}^{\FS}) \\
&= \binom{n}{m}^{-1}\int_{\mathbf p\in \oR_{m,D}} C_n  \exp\Big[m^2\Big(\oE_m(\mathbf p)- {2(m+1)\over m} \oF_m(\mathbf p)  \Big) \Big] \kappa_n(\mathbf p).
\end{align*}
Taking logarithm, we receive that
\begin{equation}\label{log-Pn-less}
\log \bP_n (H_{n,D}) \leq  \log C_n+   m^2 \sup_{\mathbf p\in (X\setminus  D)^{(m)}} \Big(\oE_m(\mathbf p)- {2(m+1)\over m} \oF_m(\mathbf p)  \Big) + \log \int_{X^{(m)}} \kappa_n.
\end{equation}
By Lemma~\ref{lem-f-m-p}, as $m\to \infty$, we have $\oF_m(\mathbf p)\geq \max U'_{\delta_{\mathbf p}} -O(\log m/m)$. Together with the uniformly boundedness of $\max U'_{\delta_{\mathbf p}}$ for all $\mathbf{p}$, we conclude that
$$ 
\sup_{\mathbf p\in (X\setminus  D)^{(m)}} 
\Big(\oE_m(\mathbf p)- 
{2(m+1)\over m} \oF_m(\mathbf p)  \Big)   
\leq  O(\log m/ m)+
 \sup_{\mathbf p\in (X\setminus  D)^{(m)}} \big  (  \oE_m(\mathbf p)  -2 \max U'_{\delta_{\mathbf p}} \big).   
$$
Since $\oE_m(\mathbf p)\leq\oE_{\Delta}(\delta_{\mathbf p})$, the last term  is bounded from above by 
$$ \sup_{\mu\in\cM(X\setminus  D)} \big(\oE_{\Delta}(\mu)- 2 \max U'_{\mu}   \big) =-\min \oI_D $$
due to Lemma \ref{lem-oE-oI}.

Now we estimate the integration of $\kappa_n$.
Using  Proposition~\ref{prop-formula}, we get 
\begin{equation}\label{integral-log-kappa-upper}
\log \int_{X^{(m)}} \kappa_n \lesssim \log \int_{X^{(m)}}
e^{O(m)}\diam(X)^{2mg}  \, i\dd z_1\wedge \dd\overline z_1 \wedge\cdots \wedge i\dd z_m\wedge \dd\overline z_m 
\lesssim m\log m. \end{equation}
Therefore,  we conclude by Lemma \ref{lem-C-n} below that
$$  \log \bP_n (H_{n,D}) \leq -m^2 \min \oI_D+  C m\log m.  $$
Notice that $C$ can be chosen independent of $D$.

\smallskip

For the lower bound, we apply Proposition~\ref{prop-oE-oF-oI} to get $\Omega_{m,D}$, so that
\begin{align*}
\bP_n (H_{n,D})&= \binom{n}{m}^{-1}\int_{\oR_{m,D}}  \Ac_m^*(V_{n}^{\FS}) \geq \binom{n}{m}^{-1}\int_{ \Omega_{m,D}  }  \Ac_m^*(V_{n}^{\FS}) \\
&= \binom{n}{m}^{-1}\int_{\mathbf p\in \Omega_{m,D}  } C_n  \exp\Big[m^2\Big(\oE_m(\mathbf p)- {2(m+1)\over m} \oF_m(\mathbf p)  \Big) \Big] \kappa_n(\mathbf p).
\end{align*}
It follows from  \eqref{ineq-oE-oF-oI} that 
\begin{equation}\label{log-Pn-great}
\log \bP_n (H_{n,D})\geq 
-\log \binom{n}{m} +\log C_n -m^2 \min \oI_D-  C_D m\log m  +  \log \int_{\mathbf p\in \Omega_{m,D} }  \kappa_n.
\end{equation}

It remains to estimate the integration of $\kappa_n$. 
Since $\Omega_{m,D}\subset \oS_{m,D}$ (see~\eqref{defn-oS-m}),  for any point $\mathbf p=p_1+\cdots+p_m\in \oS_{m,D}$ and $\mathbf q:=\oB_m(\mathbf p)=q_1+\cdots+q_g$, we have 
$$\prod_{j=1}^m\prod_{l=1}^g\dist(p_j,q_l)^2 
\geq
\prod_{j=1}^m\prod_{l=1}^g \frac{1}{m^2}=  m^{-2mg}. $$ 
Thus, applying Proposition~\ref{prop-formula} and reminding that $\Vol(\Omega_{m,D})\geq m^{-11 m}$ from Proposition~\ref{prop-oE-oF-oI}, the logarithm of the integration of $\kappa_n$ is no less than
\begin{align*}
O(m)+\log \int_{\mathbf p\in \Omega_{m,D} }  m^{-2mg}\, i\dd z_1\wedge \dd\overline z_1 \wedge\cdots \wedge i\dd z_m\wedge \dd\overline z_m  
\geq -O(m\log m). 
\end{align*}
Therefore,  we deduce from Lemma~\ref{lem-C-n} below that 
$$\log \bP_n (H_{n,D})\geq -m^2\min \oI_D   -C'_D  m\log m. $$
This completes the proof of Theorem \ref{thm-main-speed}.
\end{proof}

Some useful estimate of $C_n$ might be known in literature. For the sake of completeness, we provide a proof here.

\begin{lemma}\label{lem-C-n}
There exists a constant $c>0$ independent of $n$ and $D$ such that 
$$-c\, n\log n \leq \log C_n \leq c\, n \log n.$$
\end{lemma}

\begin{proof}
Note that $\bP_n (H_{n,\varnothing})= 1$ and $\min\oI_\varnothing =0$.  Repeating the argument in the proof of Theorem~\ref{thm-main-speed} above for $D=\varnothing$, we can deduce from~\eqref{log-Pn-less} that
$$\log C_n \geq m^2 \min\oI_\varnothing - c_1 m\log m=-c_1 m\log m$$
for some $c_1$ independent of $n$.

On the other hand, using the method of Fekete points upon the case $D=\varnothing$, we receive $\Omega_{m,\varnothing}\subset \oS_{m,\varnothing}$.  Repeating the same reasoning as~\eqref{log-Pn-great}, we can obtain that
$$\log \bP_n (H_{n,\varnothing})\geq 
-\log \binom{n}{m} +\log C_n -m^2 \min \oI_\varnothing-  c_2 m\log m  +  \log \int_{\mathbf p\in \Omega_{m,\varnothing} }  \kappa_n$$
for some $c_2$ independent of $n$.
This combining with~\eqref{integral-log-kappa-upper} yields 
$$\log C_n \leq m^2 \min\oI_\varnothing + c_3 m\log m =c_3 m\log m.$$
Hence we finish the proof.
\end{proof}

\begin{remark}\rm  \label{rmk-g=0}
For the proof of the case $g=0$, {\em i.e.}, $X=\P^1$, we can repeat the same strategy, by using a density formula of the  zeros, cf.~\cite[Proposition 3]{zei-zel-imrn}. The proof is much simpler because there is no $\mathbf q$ any more and we only need to deal with the functional $\oE_m(\mathbf p)- {2(m+1)\over m}\oF_m(\mathbf p)$ with $m=n$, where $\oF_m$ is given by
 $\oF_m(\mathbf p):=\log\big\|  e^{U'_{\delta_{\mathbf p}}}       \big\|_{L^{2m}(\omega_0)}.$
\end{remark}

\medskip

\section{Asymptotic behavior of the hole probabilities}
\label{sect: Asymptotic behavior of the hole probability}

To prove Theorem \ref{thm-main-hole}, it is enough to show the inequality for a fixed $x_0$ and all small $r$. The theorem will follow by a standard compactness argument.

Fix a point $x_0\in X$. For a radius $r>0$,  we denote $\oI_r:=\oI_{\B(x_0,r)}$ and $\nu_r:=\nu_{\B(x_0,r)}$ for abbreviation.
This section is devoted to prove the following proposition, which implies Theorem \ref{thm-main-hole} by Theorem \ref{thm-main-speed}.

\begin{proposition} \label{prop-estiamte-oI}
There exist constants $C>c>0$ independent of $r$ such that 
$$c r^4 \leq  \min \oI_r \leq Cr^4  \quad\text{for every}\quad 0<r\ll 1.$$
\end{proposition}

We decompose the proof into several lemmas below.

\begin{lemma}
For any probability measure $\mu$ supported on $X\setminus \B(x_0,r)$, one has $\oI_r(\mu )\geq c r^4$ for some $c>0$ independent of $r$ and $\mu$.
\end{lemma}

\begin{proof}[Proof 1.]
Take a local coordinate function $\phi: V\rightarrow \mathbb{C}$
 with $\phi(x_0)=0$, where $V$ contains $\overline{\B}(x_0,r_0)$ for some small $r_0>0$.
By compactness,
we can find a sufficiently small $\epsilon_0>0$
such that both  
$\omega_0$ and $\omega$ are greater than $\epsilon_0^2\cdot \phi^*\ddc |z|^2$ on $\overline{\B}(x_0,r_0)$,
 where $z$ is the  coordinate of $\C$.
For any probability measure $\mu$ supported on $X\setminus \B(x_0,r)$ with $r<r_0$,
the $\omega$-potential $U_{\mu}$ of $\mu$ satisfies the equation
$\ddc U_{\mu}=-\omega$
on  $\B(x_0,r)$.
Note that for any $r\in [0, r_0]$,
the range $\phi(\overline{\B}(x, r))$ 
contains $\overline{\D}(0, \epsilon_1\cdot r)\subset \mathbb{C}$ for some $\epsilon_1=\mathsf{const}\cdot\epsilon_0$.
We can check that
$\widehat{U}_{\mu}:=U_{\mu}\circ \phi^{-1}: \overline{\D}(0, \epsilon_1\cdot r_0)\rightarrow \R_{\leq 0}$
is strictly supharmonic
$$-\ddc \widehat{U}_{\mu}
=
-(\phi^{-1})^*
\ddc {U}_{\mu}
=
(\phi^{-1})^*\omega
\geq
(\phi^{-1})^*
(\epsilon_0^2\cdot \phi^*\ddc |z|^2)
=
\epsilon_0^2\cdot \ddc |z|^2.
$$
Using Jensen's formula (cf. e.g.~\cite{ Noguchi,MR4265173}),
for any  $0<t\leq \epsilon_1\cdot r_0$, we have
\[
 \widehat{U}_{\mu}(0)
-
\frac{1}{2\pi}
\int_{|z|=t} \widehat{U}_{\mu}
\dd \theta
=
-2\int_{x=0}^{t}
\frac{\dd x}{x}
\int_{\D_x}
\ddc \widehat{U}_{\mu}
\geq
2\int_{x=0}^{t}
\frac{\dd x}{x}
\int_{\D_x}
\epsilon_0^2\cdot \ddc |z|^2.
\]
Discarding the first term
$\widehat{U}_{\mu}(0)\leq 0$,
we obtain that
\[
-\frac{1}{2\pi}\int_{|z|=t} \widehat{U}_{\mu}
\dd \theta
\geq
2\int_{x=0}^{t}
\frac{\dd x}{x}
\int_{\D_x}
\epsilon_0^2\cdot \ddc |z|^2
=
\mathsf{const}\cdot \epsilon_0^2\cdot  t^2
\]
for any $0<t \leq \epsilon_1\cdot r_0$. 
Thus we conclude that
\[
-
\int_{\D_r}
\widehat{U}_{\mu}\,
\ddc |z|^2
=
-
\mathsf{const}
\cdot
\int_{t=0}^r
t\,\dd t \int_{|z|=t} \widehat{U}_{\mu}\dd \theta
\geq
-
\mathsf{const}
\cdot
\epsilon_0^2\cdot  r^4
\]
for any $0<r \leq \epsilon_1\cdot r_0$. Since ${U}_\mu\leq 0$,
we can bound $\cI_{r}(\mu)\geq -\int_{\B(x, r)} U_\mu \cdot \omega$ from below by
\[
-\int_{\phi(\B(x, r))} U_\mu\circ \phi^{-1} \cdot (\phi^{-1})^*\omega
\geq
-\int_{\D(0, \epsilon_1\cdot r)} \widehat{U}_\mu \cdot
\epsilon_0^2\,\ddc |z|^2
\geq
\mathsf{const}
\cdot
\epsilon_0^4\cdot  r^4
\]
for any $r\in (0, r_0)$.
\end{proof}

\begin{proof}[Proof 2.]
In the definition of $\oI_r$, $U_{\mu}$ is non-positive. It suffices to show that $\int_X U_\mu \,  \omega \geq  c r^4$. The problem is local, we can work on $\D\subset \C$.  More precisely, we will show that for any $\omega_\D$-subharmonic function $U_r$ with $U_r\leq 0$ on $\D$ and $\ddc U_r=-\omega_\D$ on $\D_r$, one has $-\int_{\D_r} U_r \,\omega_\D \geq cr^4$, where $\omega_\D:=f(z)\, i\dd z\wedge \dd\overline z$ is some fixed K\"ahler form on $\D$.

Suppose for contradiction, there exists a sequence $(U_j,r_j)$ with $\ddc U_j =-\omega_\D, U_j\leq 0$ and $\lim_{j\to \infty} r_j =0$, such that $$\lim_{j\to \infty} r^{-4} \int_{\D(0,r_j)} U_j \, \omega_\D =0.$$
Consider the dilation map $\tau_r: \D \to \D(0,r), z \mapsto rz$. Define 
$$\Phi_j:=- r_j^{-2} \tau_{r_j} ^* U_j.$$
Then on $\D$, for $j$ large, we have $\Phi_j \geq 0$  and
$$\ddc \Phi_j = r_j^{-2} \tau_{r_j}^*(-\ddc U_j)=r^{-2} \tau_{r_j}^*\omega_\D 
 = f(r_j z)\,i\dd z\wedge \dd\overline z\gtrsim i \dd z\wedge \dd \overline z.$$

 On the other hand, 
 $$\int_{\D} \Phi_j \, i \dd z\wedge \dd \overline z \lesssim \int_{\D} \Phi_j \, r^{-2} \tau_{r_j}^*\omega_\D =\int_{\D}  r^{-2}(\tau_{r_j})_* (\Phi_j ) \,  \omega_\D =- r^{-4}\int_{\D(0,r_j)}  U_j \, \omega_\D.$$
 This says the $L^1$-norm of $\Phi_j$ converges to $0$. It is not possible because $\ddc \Phi_j$ does not converge to $0$. So we get a contradiction. 
\end{proof}

To prove the upper bound, we construct a local quasi-subharmonic function first.

\begin{lemma}\label{function-Phi}
Let $\omega_\D$ be a K\"ahler form on $\D$. There exist an integer $\ell >2$ and a constant $\alpha >0$, such that for every small $0<r\ll 1$ , there exists a continuous $\omega_\D$-subharmonic function $\Phi_r$ on $\D$, such that $\Phi_r=0$ on $\D \setminus \D_{\ell r}$,  $-\alpha r^2 \leq \Phi_r \leq 0$ on $\D_{\ell r}$ and $\ddc \Phi_r =-\omega_\D$ on $\D_r$.
\end{lemma}

\begin{proof}
    Take two positive constants $\beta_1, 
    \beta_2>0$ such that
    $\beta_1 \, \ddc |z|^2 \leq \omega_\D  \leq \beta_2 \, \ddc |z|^2 $ on $\D_{1/2}$.
    Let $\ell $ be a large positive integer to be determined. Define an auxiliary function
$$\Phi_r(z):=-\beta_1(|z|-\ell r)^2$$
on $ \D_{\ell r}\setminus \D_{r}$, where $r\in (0, 1/(2\ell)]$. Direct computation shows that 
$$\ddc \Phi_r=-\beta_1 \, \ddc|z|^2+2\beta_1\ell r \,\ddc |z|\geq -\beta_1\, \ddc|z|^2
\geq
-\omega_\D$$
on $\D_{\ell r}\setminus \D_{r}\subset \D_{1/2}$. Next, we trivially extend $\Phi_r=0$ outside $\D_{\ell r}$, and we extend
$\Phi_r$ inside $\D_{r}$
as the unique solution to the elliptic equation
$\ddc \Phi_r=-\omega_{\D}$ with the boundary condition
$\Phi_r|_{\partial\D_r}=-\beta_1(\ell-1)^2 r^2$.

Now we show that $\Phi_r\leq 0$ on $\mathbb{D}_r$.
To do this, we introduce a comparison function $\Psi_r$ which solves the equation
$\ddc \Psi_r=-\beta_2\,\ddc |z|^2$ with the same boundary condition that
$\Psi_r|_{\partial\D_r}=-\beta_1(\ell-1)^2 r^2$. Noting that $\ddc (\Phi_r-\Psi_r)=-\omega_{\D}+\beta_2\,\ddc |z|^2\geq 0$ with the vanishing boundary values $(\Phi_r-\Psi_r)|_{\partial\D_r}=0$, by the maximal principle, $\Phi_r-\Psi_r\leq 0$ on $\D_r$. Thus we only need to check that $ \Psi_r\leq 0$. Indeed, it is easy to see that
\[
\Psi(z)
=
-\beta_1(\ell-1)^2 r^2
+
\beta_2\cdot
(r^2
-
|z|^2)
\leq
(-\beta_1(\ell-1)^2+\beta_2)\cdot r^2
\leq 0
\]
for $z\in\D_r$
as long as $\ell\geq \sqrt{\beta_2/\beta_1}+1$.
It is clear that $\Phi_r$ is a continuous $\omega_\D$-subharmonic function and satisfies all the desired properties with $\alpha = -\beta_1(\ell -1)^2$.
\end{proof}

\begin{lemma}
For every  small $0<r\ll 1$, there exists a probability measure $\mu$ supported outside $\B(x_0,r)$ such that $\oI_r(\mu)\leq C r^4$ for some $C>0$ independent of $r$.
\end{lemma}

\begin{proof}
Using Lemma \ref{function-Phi},  we can find an $\omega$-subharmonic function $\Phi_r$ on $X$ such that $\Phi_r=0$ on $X\setminus \B(x_0,\ell r)$, $-\alpha r^2 \leq \Phi_r \leq 0$ on $\B(x_0, \ell r)$ and $\ddc \Phi_r =-\omega$ on $\B(x_0,r)$, where $\ell,\alpha$ are independent of $r$. Let $\mu := \ddc \Phi_r +\omega$. Obviously, $\supp(\mu)\subset X\setminus \B(x_0,r)$. Observe that $\Phi_r$ is actually the $\omega$-potential of type M of $\mu$. Furthermore, since $\mu=\omega$ outside $\overline \B(x_0,\ell r)$, the mass of $\mu$ on $\overline \B(x_0,\ell r)$ is equal to $\omega(\overline \B(x_0,\ell r))$. Hence $$\oI_r(\mu) =-\int_{\overline \B(x_0,\ell r)} \Phi_r \,\dd \mu -\int_{\overline \B(x_0,\ell r)} \Phi_r \,\omega 
\leq
\alpha r^2\cdot
\int_{\overline \B(x_0,\ell r)} (\dd \mu+\omega)   \lesssim r^4.$$
This ends the proof.
\end{proof}

\bigskip

\noindent\textbf{Acknowledgements.}
Hao Wu wants to thank  Professor  Tien-Cuong Dinh for helpful discussions around the functional $\oI_D$.
Song-Yan Xie would like to 
 thank Dr. Bingxiao Liu for giving an online short course on higher dimensional hole probabilities in AMSS SCV seminar.

\bigskip

\noindent\textbf{Funding.} 
Song-Yan Xie is
partially supported by 
National Key R\&D Program of China Grants
No.~2021YFA1003100, 
No.~2023YFA1010500 and  National Natural Science Foundation of China Grant No.~12288201. 
Hao Wu is supported by the Singapore Ministry of Education Grant MOE-T2EP20121-0013.


\end{document}